\documentclass[12pt]{article}
\usepackage{IEEEtrantools}
\usepackage{mathtools, nccmath}
\usepackage{latexsym}
\usepackage{amsfonts}
\usepackage{amssymb}
\usepackage{eqnarray, amsmath,amsfonts,amsthm}
\usepackage{mathrsfs}
\usepackage{enumerate}
\usepackage{dsfont}
\usepackage{hyperref}
\usepackage{graphicx}
\usepackage{epstopdf}
\usepackage{tikz}
\usepackage[T1]{fontenc}

\usetikzlibrary{shapes,arrows.meta,positioning}
\usepackage{cleveref}
\usepackage{tikz-cd}
\usepackage{psfrag}

\usepackage{multirow}
\usepackage{tikz}
\usetikzlibrary{arrows,decorations.markings}
\usepackage{bookmark}
\usepackage{hyperref}
\hypersetup{
     colorlinks   = true,
     citecolor    = blue
}

\newcommand{\newsection}[1]{\setcounter{equation}{0}
\setcounter{dfn}{0}
\section{#1}}

\newtheorem{dfn}{Definition}[section]
\newtheorem{thm}[dfn]{Theorem}
\newtheorem{lmma}[dfn]{Lemma}
\newtheorem{ppsn}[dfn]{Proposition}
\newtheorem{crlre}[dfn]{Corollary}
\newtheorem{xmpl}[dfn]{Example}
\newtheorem{rmrk}[dfn]{Remark}

\newtheorem*{thmA}{Theorem A}
\newtheorem*{thmB}{Theorem B}

\newcommand{\bbc}{\mathbb{C}}
\newcommand{\bbz}{\mathbb{Z}}
\newcommand{\bbn}{\mathbb{N}}

\def \qed { \mbox{}\hfill
$\Box$\vspace{1ex}}

\providecommand{\Dref}[1]{Definition~\ref{#1}}
\providecommand{\Tref}[1]{Theorem~\ref{#1}}
\providecommand{\Lref}[1]{Lemma~\ref{#1}}
\providecommand{\Pref}[1]{Proposition~\ref{#1}}
\providecommand{\Yref}[1]{Corollary~\ref{#1}}
\providecommand{\Sref}[1]{Section~\ref{#1}}
\providecommand{\Eref}[1]{(\ref{#1})} 

\title{Sections and Chapters}
\let\oldthebibliography\thebibliography
\renewcommand{\thebibliography}[1]{
  \oldthebibliography{#1}
  \setlength{\itemsep}{0pt}
  \setlength{\parsep}{0pt}
  \setlength{\parskip}{0pt}
}

\begin{document}

\tikzset{->-/.style={decoration={
  markings,
  mark=at position #1 with {\arrow{>}}},postaction={decorate}}}
  \tikzset{-<-/.style={decoration={
  markings,
  mark=at position #1 with {\arrow{<}}},postaction={decorate}}}

\author{\sc{Vibhor Bhatt,\,\,Satyajit Guin,\,\,Bipul Saurabh}}
\title{Quantum Gromov--Hausdorff Convergence for Extensions of $C^*$-Algebras}
\maketitle


\begin{abstract}
We investigate the lifting of quantum Gromov--Hausdorff convergence through Toeplitz type $C^*$-algebra extensions by stable ideals in the framework of noncommutative metric geometry. Working with the spectral metric space construction of Hawkins and Zacharias (Comm. Math. Phys. 350 (2017), 475–506), we consider a sequence of complete sub-operator systems of the quotient or the unital $C^*$-algebra underlying the stable ideal, converging in the quantum Gromov--Hausdorff distance. We study whether this induces a corresponding convergent sequence of complete sub-operator systems of the extension. To address this problem, we construct complete sub-operator systems of the extension associated with those of the quotient and the unital $C^*$-algebra underlying the stable ideal. We also introduce the notion of unital $2$-contractive approximation together with its Toeplitz type refinement to provide the compatibility required by the commutator structure of the Dirac operator on the extension. We prove that, under this approximation hypothesis on the convergent sequence in the quotient or the unital $C^*$-algebra underlying the stable ideal, quantum Gromov--Hausdorff convergence lifts to the extension.
\end{abstract}

\smallskip


\hypersetup{linkcolor=blue}
\tableofcontents


\newsection{Introduction}\label{Sec1}

The central ingredient in Connes' noncommutative geometry \cite{Con} is the notion of a spectral triple. From the point of view of Kasparov theory, spectral triples can be regarded as abstract Dirac-type elliptic operators associated to a $C^*$-algebra $A$: they are Baaj--Julg $(A,\bbc)$-cycles \cite{BJ}, the unbounded Kasparov modules in $KK$-theory \cite{K,B}, and determine $K$-homology classes via the corresponding bounded Kasparov cycles \cite{HR}. When equipped with additional regularity conditions, spectral triples provide a powerful framework for encoding geometric information in a noncommutative setting \cite{CM}. In this paper, we focus on the metric structures induced by spectral triples. Given a spectral triple $(\mathcal{A},\mathcal{H},\mathcal{D})$ over a unital $C^*$-algebra $A$ that is nondegenerate (i.e. $[\mathcal{D},a]=0$ only for $a\in\bbc\,. 1_A$), Connes \cite{C,Con} showed the following formula
\begin{center}
$\mathrm{dist}_{\mathcal{D}}(\phi,\psi):=\sup\{|\phi(a)-\psi(a)|\,:\,a\in\mathcal{A},\,\|[\mathcal{D},a]\|\le 1\}$
\end{center}
defines an extended metric (possibly $\infty$-valued) on the state space $\mathcal{S}(A)$ of the $C^*$-algebra $A$. For the canonical Dirac triple on a closed spin manifold, the restriction of the metric to the point evaluation measures coincides with the geodesic metric.
\smallskip

This metric aspect of spectral triples inspired Rieffel to systematically develop the theory of compact quantum metric spaces (CQMS), a generalization of compact metric spaces, or Lipschitz functions, to order-unit spaces with the quantum Gromov--Hausdorff distance as a central tool for quantifying the proximity between such spaces \cite{Rie1,Rie2,Rie3,Rie4}. We note the existence of several other distances due to Kerr \cite{Kerr,KL}, Li \cite{Li}, and Latr\'emoli\`ere \cite{L1,L2}, but we do not consider them here. Although Rieffel's initial CQMS framework is formulated in the setting of order-unit spaces, it can be recast within the \emph{operator system} framework of Connes and van Suijlekom \cite{CS}, and so can the notion of spectral triples (called operator system spectral triples), as is the recent trend \cite{AKK1,AKK2,CS1,Rie5,Rie6,AKK3,Kaad,AuK,S2}. Every operator system has an underlying order-unit space given by its self-adjoint part. We work in this generalized framework.
\smallskip

To work in this setting precisely, we now fix the relevant objects. Throughout the article, $X$ denotes a (concrete) complete operator system that is a closed subspace in a unital $C^*$-algebra $A$ stable under adjoint operation and containing the unit $1_A$. A complete sub-operator system of a complete operator system $X$ is a closed linear subspace stable under adjoint operation and containing the unit. Every unital $C^*$-algebra $A$ is in particular a complete operator system, so the definitions below specialize to the case of $X=A$. A Lipschitz seminorm on a complete operator system $X$ is a seminorm $L:\mathcal{X}\to\mathbb{R}^+$ defined on a dense unital involutive subspace $\mathcal{X}$ of $X$ such that $L(a^*)=L(a)$ for each $a\in\mathcal{X}$ and $L(1)=0$. A Lipschitz seminorm $L$ is called \emph{nondegenerate} if the set $\{a\in\mathcal{X}:L(a)=0\}$ contains only multiples of identity. A Lipschitz seminorm $L$ on $X$ determines an extended metric on the state space $\mathcal{S}(X)$:
\begin{center}
$d_{\mathcal{X},L}(\phi,\psi):=\sup\{|\phi(a)-\psi(a)|\,:\,a\in\mathcal{X},\,L(a)\le 1\}.$
\end{center}

This construction is, in fact, reversible, which is what makes the seminorm and metric pictures interchangeable. A metric $d$ on $\mathcal{S}(X)$ induces a nondegenerate seminorm on $X$:
\[
L_d(a):=\sup\left\{\frac{|\phi(a)-\psi(a)|}{d(\phi,\psi)}\,:\,\phi,\psi\in\mathcal{S}(X),\,\phi\neq\psi\right\}.
\]
If $L$ is a Lipschitz seminorm on $X$, then so is $L_{d_{\mathcal{X},L}}$. Moreover, if $L$ is \emph{lower semicontinuous}, then $L=L_{d_{\mathcal{X},L}}$. A CQMS is a pair $(X,L)$ where $X$ is a complete operator system equipped with a Lipschitz seminorm $L$ with dense domain $\mathcal{X}\subseteq X$ such that $d_{\mathcal{X},L}$ induces the weak$\,^\ast$-topology on $\mathcal{S}(X)$ (see \Sref{Sec2} for detailed discussion).
\smallskip

If an operator system spectral triple $(\mathcal{X},\mathcal{H},\mathcal{D})$ over $X$ satisfies the nondegeneracy condition, that is, $[\mathcal{D},x]=0$ if and only if $x\in\mathbb{C}.1_X$, then the formula $L_{\mathcal{D}}(x):=\|\mathrm{cl}([\mathcal{D},x])\|$ defines a lower semicontinuous Lipschitz seminorm on $X$ with dense domain $\mathcal{X}$. However, determining whether Connes' extended metric $d_{\mathcal{X},L_{\mathcal{D}}}$ on $\mathcal{S}(X)$ is indeed a metric that induces the weak$\,^\ast$-topology on $\mathcal{S}(X)$ is a genuinely hard analytic problem rather than a routine formality. Spectral triples inducing CQMS structures are called spectral metric spaces \cite{BMR}. Examples of spectral metric spaces include group $C^*$-algebras \cite{Rie7,OR,CR} and crossed products \cite{BMR,AKK3}, Connes-Landi $\theta$-deformations \cite{Li0}, AF $C^*$-algebras \cite{CI0}, Bunce-Deddens algebras, noncommutative tori, and some of their generalizations \cite{HSWZ,Kl}, as well as quantum groups of rapid decay \cite{BVZ}, standard Podle\'s spheres \cite{AK} and quantum projective spaces \cite{MK}, quantum $SU(2)$ \cite{KK}, and Toeplitz type $C^*$-extensions \cite{HZ} which will be of our interest.
\smallskip

Spectral metric spaces are but one source of CQMS; approximation of a CQMS by other, often finite-dimensional, CQMS is an important aspect of the theory, made precise by the quantum Gromov--Hausdorff distance introduced above, and it applies just as much to spectral metric spaces as to CQMS in general. Rieffel himself showed convergence of matrix algebras to the sphere \cite{Rie4}, and, under suitable conditions, that of Fourier truncations of compact quantum groups \cite{Rie6}. Other examples include approximations of quantum tori by fuzzy tori \cite{L0}, and the Podle\'s sphere $S_q^2$ by $q$-analogue of fuzzy spheres \cite{AKK1}. The need for approximations arises naturally from both mathematical and physical considerations \cite{Rie4,Rie5}. Recently, Connes and van Suijlekom introduced the notion of \emph{spectral truncation} in the operator system framework \cite{CS}; this line of investigation is developing rapidly \cite{S,CS1,LS,GS}, and new convergence results have been obtained in this setting as well.
\smallskip

The central object of this paper is certain extensions of unital $C^*$-algebras by stable ideals, together with their associated compact quantum metric structures. We work with the so called Toeplitz type $C^*$-extensions investigated by Hawkins and Zacharias \cite{HZ}, which generalizes earlier work of Christensen and Ivan \cite{CI}, and with the construction of compact quantum metric structures on these extensions. The central concern of this paper is:
\smallskip

``Given a Toeplitz type $C^*$-extension equipped with a compact quantum metric structure, together with a sequence of complete sub-operator systems converging in the quantum Gromov--Hausdorff distance in either the quotient or the unital $C^*$-algebra underlying the stable ideal, can this convergence be lifted to the extension?''
\smallskip

A primary difficulty is the commutator structure of the Hawkins--Zacharias Dirac operator on the extension that contains several interacting components arising from the quotient, the stable ideal, and the Toeplitz projection. Consequently, the approximation of the quotient or ideal does not immediately imply the approximation of the extension. We make these precise and describe our work in the following.
\medskip

\noindent\textbf{Toeplitz type extensions and their quantum metric structures.} Extensions play a crucial role in $C^*$-algebra theory. Given separable unital $C^*$-algebras $A\mbox{ and }B$ and faithful unital representations $\pi_A:A\to B(H_A)$ and $\pi_B:B\to B(H_B)$, the $C^*$-extension
\begin{IEEEeqnarray}{lCl}\label{ext}
0\longrightarrow\mathcal{K}(H_A)\otimes B\longrightarrow E\longrightarrow A\longrightarrow 0
\end{IEEEeqnarray}
where $\mathcal{K}(H_A)$ is the space of compact operators on $H_A$, is called \emph{Toeplitz type} if there exists an infinite-dimensional projection $P\in B(H_A)$ such that $[P\,,\,\pi_A(a)]\in\mathcal{K}(H_A),\,E\,\cong\,\mathcal{K}(PH_A)\otimes \pi_B(B)+P\pi_A(A)P\otimes\bbc .1_B$, and $\mathcal{K}(PH_A)\otimes \pi_B(B)\cap P\pi_A(A)P\otimes\bbc .1_B=\{0\}$ (see formally \Dref{Toeplitz extension}). The triple $(\pi_A,\pi_B,P)$ is then referred to as a \emph{Toeplitz triple} for the extension. If $(\mathcal{A},H_A,\mathcal{D}_A)$ is a spectral triple for $A$ such that $\mathcal{D}_A$ and $P$ commute and $[P,\pi_A(\mathcal{A})]\in \mathcal{C}(H_A)$, the space of differentiable compacts, then the quadruple $(\mathcal{A},H_A,\mathcal{D}_A,P)$ is said to be of \emph{Toeplitz type}. See \Sref{Sec3.1} and \ref{Sec3.2} for details.
\smallskip

Starting with a Toeplitz type quadruple $(\mathcal{A},H_A,\mathcal{D}_A,P)$ for the quotient $A$ and a spectral triple $(\mathcal{B},H_B,\mathcal{D}_B)$ for $B$, Hawkins and Zacharias constructed a spectral triple $(\mathcal{E},H_E,\mathcal{D}_E)$ with the desired spectral dimension for the extension $E$ (see \Sref{Sec3.3}). More importantly, they showed that if the spectral triple of Toeplitz type $(\mathcal{A},H_A,\mathcal{D}_A,P)$ satisfies the mild assumption that $\mathcal{D}_A$ is $P$-injective (i.e., $\mathrm{ker}(P\mathcal{D}_A)\,\cap\, PH_A=\{0\}$), and if both $(A,L_{\mathcal{D}_A})$ and $(B,L_{\mathcal{D}_B})$ inherit CQMS structures, then so does $(E,L_{\mathcal{D}_E})$. In other words, they construct a compact quantum metric structure on the extension starting with such structures on its constituents.
\medskip

\noindent\textbf{Main Results.} Now, let us discuss our findings in this article. Our goal here is not to be overly meticulous in the exact formulations of the theorems; instead, we seek to convey the overall picture. This allows us to avoid delving into the notational details and technicalities they entail.
 
Since every unital $C^*$-algebra is by definition a complete operator system, the main outcomes of \cite{HZ} discussed above fit into the operator system framework. Let $(\mathcal{X},H_X,\mathcal{D}_X)$ be a spectral triple for a complete operator system $X$ such that $(X,L_{\mathcal{D}_X})$ is a CQMS. Our first contribution is the introduction of the notion of unital $2$-contractive approximation adapted to general spectral metric spaces. Associated to a pair $(X,\{X_n\}_{n\in\bbn})$ consisting of $X$ and a sequence $\{X_n\}_{n\in\bbn}$ of complete sub-operator systems of it, a unital $2$-contractive approximation consists of a sequence of unital $2$-contractive maps on $B(H_X)$ satisfying certain conditions (\Dref{def 1}). Moreover, if $(\mathcal{X},H_X,\mathcal{D}_X,P)$ is a Toeplitz type quadruple for $X$, we introduce unital $2$-contractive approximation of Toeplitz type (\Dref{def 2}). Note that unital $2$-contractive maps acting on any $B(H)$ are not necessarily ucp maps, and hence the notion of $2$-contractive approximation is weaker than a similar notion of ucp approximation but sufficient for our purposes. These approximation tools are specifically designed to control the commutator seminorm associated with the Dirac operator on the extension. Consequently, we have the following main results (see Sections \ref{Sec5}--\ref{Sec6}).
\begin{thmA}
Consider a Toeplitz type extension \Eref{ext} such that $(E,L_{\mathcal{D}_E})$ is a spectral metric space by the Hawkins--Zacharias construction. Let $\{A_n\}_{n\in\bbn}$ be a sequence of complete sub-operator systems of $A$ such that each $A_n\cap\mathcal{A}$ is dense in $A_n$, and the sequence $\{(A_n, L_{\mathcal{D}_A}|_{A_n})\}_{n\in\bbn}$ of CQMS converges to the CQMS $(A, L_{\mathcal{D}_A})$ in the quantum Gromov--Hausdorff distance. For each $n\in\bbn$, let $Q_n=\sum_{k=1}^nP_k$ where $P_k$'s are the spectral projections of $\left(P\mathcal{D}_A|_{PH_A}\right)^{-1}\in\mathcal{K}(PH_A)$. One has the following:
\begin{enumerate}[$(i)$]
\item The subspaces $$E_n:= Q_n\mathcal{K}(PH_A)Q_n\otimes B+ PA_nP\otimes\bbc$$ are complete sub-operator systems of $E$, and each $E_n\cap\mathcal{E}$ is dense $\ast$-invariant in $E_n$.
\item If there exists a unital $2$-contractive approximation of Toeplitz type associated to the pair $(A,\{A_n\}_{n\in\bbn})$, then the sequence $\{(E_n,L_{\mathcal{D}_E}|_{E_{n}})\}_{n\in\bbn}$ of CQMS converges to the CQMS $(E, L_{\mathcal{D}_E})$ in the quantum Gromov--Hausdorff distance.
\end{enumerate}
\end{thmA}

Thus, controlled approximations of the quotient lift to approximations for the extension.

\begin{thmB}
Consider a Toeplitz type extension \Eref{ext} such that $(E,L_{\mathcal{D}_E})$ is a spectral metric space by the Hawkins--Zacharias construction. Let $\{B_n\}_{n\in\bbn}$ be a sequence of complete sub-operator systems of $B$ such that each $B_n\cap\mathcal{B}$ is dense in $B_n$, and the sequence $\{(B_n, L_{\mathcal{D}_B}|_{B_n})\}_{n\ge 1}$ of CQMS converges to the CQMS $(B, L_{\mathcal{D}_B})$ in the quantum Gromov--Hausdorff distance. One has the following:
\begin{enumerate}[$(i)$]
\item The subspaces $$G_n:= \mathcal{K}(PH_A) \otimes B_n+ PAP\otimes\bbc$$ are complete sub-operator systems of $E$, and each $G_n\cap\mathcal{E}$ is dense $\ast$-invariant in $G_n$.
\item Assume that there exists a Toeplitz type sequence of $2$-contractive maps on the quotient $A$. Then, if there exists a unital $2$-contractive approximation associated to the pair $(B,\{B_n\}_{n\in\bbn})$, the sequence $\{(G_n,L_{\mathcal{D}_E}|_{G_{n}})\}_{n\in\bbn}$ of CQMS converges to the CQMS $(E, L_{\mathcal{D}_E})$ in the quantum Gromov--Hausdorff distance.
\end{enumerate}
\end{thmB}
Thus, controlled approximations for the unital $C^*$-algebra $B$ underlying the stable ideal $\mathcal{K}(H_A)\otimes B$ induce approximations for the extension.

We emphasize that the $2$-contractive approximation hypothesis in Theorems A and B is imposed solely on the approximating sequence in $A\mbox{ or in }B$, and not on the extension $E$.
\smallskip

The proofs of Theorems A and B combine three ingredients: analysis of the Dirac operator constructed in \cite{HZ}, unital $2$-contractive approximation and its Toeplitz type refinement, and construction of auxiliary complete sub-operator systems of the extension along the horizontal and vertical directions. The case of simultaneous approximation from the quotient and the ideal requires formulating a ``joint'' notion of unital $2$-contractive approximation of Toeplitz type, and we shall address this elsewhere.
\medskip

We conclude the Introduction with one important remark. In this article, we work with the domain $\mathcal{E}$ of the Lip-norm $L_{\mathcal{D}_E}$ on the extension $E$ that comes from the spectral triple $(\mathcal{E},H_E,\mathcal{D}_E)$. One can also consider the associated Lipschitz algebra $\mathrm{Lip}_{\mathcal{D}_E}(E)$ of the spectral triple as the domain of the Lip-norm, denoted by $L_{\mathcal{D}_E}^{\max}$. However, the tools and techniques used here to investigate the quantum Gromov--Hausdorff convergence do not seem sufficient to handle $\mathrm{Lip}_{\mathcal{D}_E}(E)$. This situation is somewhat reminiscent of the case of the Podle\'s sphere studied in \cite{AKK2}, where techniques from von Neumann algebras are used to a greater extent than $C^*$-algebraic methods to handle the Lipschitz algebra. We expect that the convergence results established here also hold in the Lipschitz algebra setting as well, although we do not yet have a complete proof.
\medskip

The organization of the article is as follows. In \Sref{Sec2} we recall preliminary material about operator systems and quantum Gromov--Hausdorff distance that are essential in the context of the paper. In \Sref{Sec3}, we briefly recall Toeplitz type $C^*$-extensions and discuss the compact quantum metric structures on this class of extensions. \Sref{Sec4} develops the notion of unital $2$-contractive approximation and its Toeplitz type refinement. This is followed by the construction in \Sref{Sec5} of complete sub-operator systems in the extension from those on the constituents. Finally, \Sref{Sec6} forms the main part of the article, where we investigate quantum Gromov--Hausdorff convergence for Toeplitz type extensions.
\medskip

\noindent\textbf{\textit{Acknowledgement.}} S. Guin thanks Vern Paulsen for providing examples for a question on unital $2$-contractive maps and ucp maps.


\newsection{Preliminaries}\label{Sec2}

In this section, we recall the relevant material on compact quantum metric spaces, generalized in the operator system framework.

\subsection{Operator systems and compact quantum metric spaces}\label{Sec2.1}

We begin with the definition of (concrete) complete operator systems.
\begin{dfn}
A (concrete) operator system $\mathcal{X}$ is a subspace of a unital $C^*$-algebra $\mathscr{A}$ such that $\mathcal{X}$ is stable under the adjoint operation and contains the unit $1_{\mathscr{A}}$. We call it a complete operator system if additionally it is closed.
\end{dfn}

\begin{dfn}
A linear subspace $\mathcal{Y}$ of an operator system $\mathcal{X}$ is called a sub-operator system if $1_{\mathcal{X}}\in \mathcal{Y}$ and $y^{*} \in \mathcal{Y}$ for all $y \in \mathcal{Y}$.

A sub-operator system $Y$ of a complete operator system $X$ is called a complete sub-operator system if $Y$ is closed in $X$.
\end{dfn}

Throughout the article, notation $X,Y$ etc. will denote complete operator systems, whereas calligraphic letters $\mathcal{X,Y}$ etc. stand for operator systems. We shall only work with concrete operator systems and hence drop the word `concrete' for brevity.

A complete operator system $X$ has an associated state space $\mathcal{S}(X)$ consisting of all positive linear functionals that preserve the unit. A state on $X$ automatically has norm $1$, and the state space $\mathcal{S}(X)$ therefore becomes a compact Hausdorff space for the weak$\,^*$-topology. It is known that any positive map $\Phi:X\to Y$ into another operator system $Y$ need not be a contraction, but it is bounded with $\|\Phi\|\le 2\|\Phi(1_X)\|$ \cite[Propn. 2.1]{Paul}. If $\Phi$ is completely positive, then $\Phi$ is completely bounded and $\|\Phi\|=\|\Phi\|_{cb}=\|\Phi(1_X)\|$ \cite[Propn. 3.6]{Paul}. If $\Phi$ is assumed to be unital and contractive, then it is automatically positive.

The self-adjoint part of a complete operator system $X$ 
\[
X_{sa}:= \{x \in X : x= x^*\}
\]
forms a real vector space, which is an order-unit space equipped with the order-unit $1_X$ and the partial order inherited from the ambient $C^*$-algebra $A$. In particular, positivity in $X_{sa}$ is determined by the positivity relation in $A$.

Any state on $X_{sa}$ extends uniquely to a complex-valued state on $X$. Indeed, given a state $\varphi :X_{sa} \rightarrow \mathbb{R}$, the extension $\widetilde{\varphi}: X \rightarrow \mathbb{C}$ is defined by evaluating $\varphi$ separately on the real and imaginary parts of an element $x \in X$:
\[\widetilde{\varphi}(x):=\varphi(\Re(x))+i\,\varphi(\Im(x))\,.\]

\begin{dfn}[\cite{KK}]\label{CQMS}
A compact quantum metric space is a complete operator system $X$ together with a seminorm $L:X\rightarrow [0,\infty]$ satisfying the following:
\begin{enumerate}[$(i)$]
\item $L(x)=0$ if and only if $x\in\bbc$.
\item The set $\mathrm{Dom}(L):=\{x\in X:L(x)<\infty\}$ is dense in $X$ and $L$ satisfies that $L(x^*)=L(x)$ for all $x\in X$.
\item The function $d_L(\mu,\nu):=\sup\{|\mu(x)-\nu(x)|\,:\,L(x)\leq 1\}$ defines a metric on the state space $\mathcal{S}(X)$ that metrises the weak$\,^\ast$-topology on it.
\end{enumerate}
\end{dfn}
A seminorm $L$ on $X$ satisfying conditions $(i)$ and $(ii)$ above is referred to as a nondegenerate Lipschitz seminorm, and if $(X,L)$ defines a compact quantum metric space, then $L$ is called a Lip-norm on $X$. The metric $d_L$, sometimes written as $\mathrm{mk}_L$, is referred to as the Monge-Kantorovich metric.

The following result of Rieffel is an indispensable tool in verifying compact quantum metric structure.

\begin{thm}[\cite{Rie1}]
Let $X$ be a complete operator system and $L: X \rightarrow [0, \infty]$ be a seminorm satisfying (i) and (ii) from Definition \ref{CQMS}. Then $(X, L)$ is a compact quantum metric space if and only if $(X, L)$ has finite diameter and the image of the set $\{x \in X\,:\,L(x)\leq 1\mbox{ and }\|x\| \leq 1\}$ under the quotient map $X \rightarrow X/\mathbb{C}$ is totally bounded for the quotient norm.
\end{thm}

Now, let us recall the definition of spectral triple, the fundamental ingredient in noncommutative geometry, generalized in the operator system framework.

\begin{dfn}[\cite{CS,LS}]\label{operator system spectral triple}
An operator system spectral triple is a triple $(\mathcal{X},\mathcal{H}, \mathcal{D})$ where $\mathcal{X}$ is a dense subspace of a complete operator system $X$ in $B(\mathcal{H})$, $\mathcal{H}$ is a Hilbert space and $\mathcal{D}$ is an unbounded self-adjoint operator acting on $\mathcal{H}$ with compact resolvent such that $T\,(\mathrm{dom}(\mathcal{D}))\subseteq\mathrm{dom}(\mathcal{D})$ and $[\mathcal{D},T]$ extends to a bounded operator for all $T\in \mathcal{X}$.
\end{dfn}

Any spectral triple $(\mathcal{X},\mathcal{H},\mathcal{D})$ over $X$ induces a Lipschitz seminorm $L_\mathcal{D}:X\rightarrow[0,\infty]$ given by $L_\mathcal{D}(x):=\|\mathrm{cl}([\mathcal{D},x])\|$ for $x\in \mathcal{X}$. If $L_\mathcal{D}$ metrises the weak$\,^\ast$-topology on the state space $\mathcal{S}(X)$, one often says that the spectral triple $(\mathcal{X},\mathcal{H},\mathcal{D})$ is a spectral metric space, a notion originally dating back to \cite{BMR}.

It is always an interesting and nontrivial question whether a spectral triple inherits a compact quantum metric structure, but the verification is not trivial and, in fact, often quite involved.

\subsection{Quantum Gromov--Hausdorff distance}\label{Sec2.2}

Quantum Gromov--Hausdorff distance for CQMS, introduced by Rieffel in \cite{Rie3}, is the noncommutative counterpart of the classical Gromov–Hausdorff distance for compact metric spaces. We briefly recall it here following \cite{KK} in the operator system framework.

Let $(X, L_{X})$ and  $(Y, L_{Y})$ be compact quantum metric spaces in the sense of \Dref{CQMS}. Recall that the self-adjoint part $X_{sa}$ of a complete operator system $X$ is an order-unit space.
\begin{dfn}[\cite{KK}]
A Lipschitz seminorm $L: X\oplus Y\rightarrow [0, \infty]$ is said to be admissible if the pair $(X\oplus Y,L)$ is a compact quantum metric space, $\mathrm{Dom}(L)=\mathrm{Dom}(L_X)\oplus\mathrm{Dom}(L_Y)$ and the quotient seminorms induced by $L_{sa}$ via the coordinate projections $\mathrm{Dom}(L)_{sa}\to\mathrm{Dom}(L_X)_{sa}$ and $\mathrm{Dom}(L)_{sa}\to\mathrm{Dom}(L_Y)_{sa}$ agree with $(L_X)_{sa}$ and $(L_Y)_{sa}$ respectively.
\end{dfn}

Whenever $L$ is an admissible Lipschitz seminorm, it follows that the coordinate projections $X\oplus Y\to X$ and $X\oplus Y\to Y$ induce isometries $\mathcal{S}(X)\to\mathcal{S}(X\oplus Y)$ and $\mathcal{S}(Y)\to\mathcal{S}(X\oplus Y)$ of the state spaces equipped with the Monge--Kantorovic metrics coming from the relevant Lip-norms.

These embeddings allow us to measure the Hausdorff distance between the state spaces $\mathcal{S}(X)$ and $\mathcal{S}(Y)$ with respect to the Monge--Kantorovic metric on the state space $\mathcal{S}(X\oplus Y)$. Denoting this Hausdorff distance by $\mathrm{dist}_{H}^{\rho_{L}}\big(\mathcal{S}(X), \mathcal{S}(Y)\big)\in[0,\infty)$, one defines the following.
\begin{dfn}[\cite{Rie3}]
The quantum Gromov--Hausdorff distance between $(X, L_X)$ and $(Y, L_Y)$ is defined as
\[
\mathrm{dist}_Q\big((X, L_X), (Y, L_Y)\big):=\inf\big\{\mathrm{dist}_{H}^{\rho_{L}}\big(\mathcal{S}(X)\,,\,\mathcal{S}(Y)\big)\,:\,L:X \oplus Y\rightarrow[0,\infty]\,\,\mathrm{admissible}\big\}.
\]
\end{dfn}
The quantum Gromov–Hausdorff distance quantifies how closely the state spaces of two compact quantum metric spaces can be positioned within a larger ambient compact quantum metric space while preserving their respective metric structures.
\smallskip

We refer to \cite{KK} for preliminary results on quantum Gromov--Hausdorff distance in the operator system framework, and record the following result which will be useful for our purposes.
\begin{lmma}[\cite{Rie3,AKK1}]\label{Kaad}
Let $(X,L)$ be a compact quantum metric space and $Y\subseteq X$ be a complete sub-operator system equipped with the seminorm $K:Y\longrightarrow[0,\infty]$ with dense $\ast$-invariant domain $\mathrm{Dom}(K)\subseteq\mathrm{Dom}(L)$ such that $K(y)=L(y)$ for all $y\in\mathrm{Dom}(K)$. Let $\varepsilon>0$. If for every $x\in X$, there exists $y\in Y$ such that $K(y)\leq L(x)$ and $\|x-y\|\leq\varepsilon\, L(x)$, then $\mathrm{dist}_Q\big((X,L),(Y,K)\big)\leq\varepsilon$, where $\mathrm{dist}_Q$ denotes the quantum Gromov--Hausdorff distance.
\end{lmma}


\newsection{Toeplitz type $C^*$-extensions and spectral triples}\label{Sec3}

We recall some necessary facts about Toeplitz type $C^*$-extensions from \cite{HZ}, and briefly discuss the compact quantum metric structures arising from spectral triples.

\subsection{Toeplitz type $C^*$-extensions}\label{Sec3.1}

Let $A\mbox{ and }B$ be unital separable $C^*$-algebras faithfully represented on separable Hilbert spaces $H_A\mbox{ and }H_B$ via $\pi_A\mbox{ and }\pi_B$ respectively. Consider a short exact sequence of $C^*$-algebras
\begin{IEEEeqnarray}{lCl}\label{Toeplitz type extension}
0\longrightarrow\mathcal{K}(H_A)\otimes B\longrightarrow E\longrightarrow A\longrightarrow 0
\end{IEEEeqnarray}
where the ideal $\mathcal{K}(H_A)\otimes B$ is essential, i.e. it has non-zero intersection with any other ideal $I\subseteq E$. The extension is characterized by a $\ast$-homomorphism $\Psi:A\to\mathcal{Q}_B$, the Busby invariant, where $\mathcal{Q}_B:=\mathcal{L}_B/\mathcal{K}_B$ is the generalized Calkin algebra with respect to the $C^*$-algebra $B$. If $q_B:\mathcal{L}_B\to\mathcal{Q}_B$ denotes the quotient map, we have
\[
E\,\cong\,\mathcal{L}_B\oplus_{(q_B,\Psi)}A:=\{(x,a)\in\mathcal{L}_B\oplus A\,:\,q_B(x)=\Psi(a)\},
\]
and the following diagram commutes:
\begin{center}
\begin{tikzcd}[row sep=small, column sep=tiny] 
0 &\longrightarrow& \mathcal{K}_B \arrow[d, "\,\pi|_{\overline{B}}"]  & \xhookrightarrow{\quad} & E \arrow[d, "\,\pi"]  & \longrightarrow & A \arrow[d, "\,\Psi"] &\longrightarrow& 0\\ 
0 &\longrightarrow& \mathcal{K}_B & \xhookrightarrow{\quad} & \mathcal{L}_B & \longrightarrow & \mathcal{Q}_B &\longrightarrow& 0
\end{tikzcd}
\end{center}
where $\pi:E\to\mathcal{L}_B$ is given by $\pi(x,a)=x$. We consider those extensions in which the Busby invariant $\Psi$ admits a unital completely positive lift, i.e. there is a ucp map $s:A\to\mathcal{L}_B$ such that $q_B\circ s=\Psi$. Such extensions are called \emph{semisplit}. In this case, there is a faithful representation $\rho:A\to M_2(\mathcal{L}_B)\cong\mathcal{L}_B$ and an orthogonal projection $P\in M_2(\mathcal{L}_B)\cong\mathcal{L}_B$ such that $[P\,,\,\rho(a)]\in M_2(\mathcal{K}_B)$ and $\rho_{11}(a)=s(a)$ for each $a\in A$, where
\[
\rho= \begin{pmatrix}
\rho_{11} & \rho_{12}\\
\rho_{21} & \rho_{22}
\end{pmatrix};\quad P=\begin{pmatrix}
1 & 0\\
0 & 0
\end{pmatrix}\,.
\]
The pair $(\rho,P)$ is called a Stinespring dilation of $s:A\to\mathcal{L}_B\cong\mathcal{M}(\mathcal{K}\otimes B)$. We note that the existence of such a map is not automatic, unless $A$ is nuclear.

\begin{dfn}\label{Toeplitz extension}
The extension in \Eref{Toeplitz type extension} is said to be of \emph{Toeplitz type} if there exists an infinite dimensional projection $P\in\mathcal{B}(H_{A})$ such that
\begin{enumerate}[$(i)$]
\item $[P, a] \in \mathcal{K}(H_A)$,
\item $E\cong \mathcal{K}(PH_{A})\otimes B + PAP\otimes\bbc .1_B$,
\item $\mathcal{K}(PH_{A})\otimes B\,\cap\, PAP\otimes\bbc .1_B=\{0\}$.
\end{enumerate}
The tuple $(\pi_A,\pi_B,P)$ is referred to as a Toeplitz triple for the extension.
\end{dfn}
Toeplitz type extensions form a large class, see \cite[Sections $3.2\mbox{ and }6$]{HZ}.


\subsection{Toeplitz type spectral triples}\label{Sec3.2}

\begin{dfn}[\cite{HZ}]\label{differentiable compacts}
Given a spectral triple $(\mathcal{A}, \mathcal{H}, \mathcal{D})$ over $A$, the dense subalgebra of $\mathcal{D}$-differentiable compacts, denoted by $\mathcal{C(H)}$, is the algebra of all compact operators $T\in\mathcal{K}(\mathcal{H})$ such that
\begin{enumerate}[$(i)$]
\item $T(\mathrm{dom}(\mathcal{D}))\subseteq\mathrm{dom}(\mathcal{D})$;
\item the operators $T\mathcal{D}:\mathrm{dom}(\mathcal{D})\to \mathcal{H}$ and $\mathcal{D}T:\mathrm{dom}(\mathcal{D})\to \mathcal{H}$ are closable;
\item the closures, $\mathrm{cl}(T\mathcal{D})\mbox{ and }\mathrm{cl}(\mathcal{D}T)$ respectively, are bounded operators.
\end{enumerate}
\end{dfn}
Using the above conditions, one can, in fact, write an even spectral triple on the algebra of compact operators $\mathcal{K}(\mathcal{H})$:
\[
\left(\mathcal{C(H)}\,,\,\mathrm{id}\oplus 0\,,\,\begin{pmatrix}
0 & \mathcal{D}\\
\mathcal{D} & 0
\end{pmatrix}\right).
\]
Note that $\mathcal{C(H)}$ is a Banach $\ast$-algebra when equipped with the norm
\[
\|x\|_1:=\|x\|+\max\{\|x\mathcal{D}\|\,,\,\|\mathcal{D}x\|\}\,.
\]
Thus, $\mathcal{C(H)}$ plays the role of the `differentiable' elements with respect to this choice of the spectral triple.

\begin{dfn}[\cite{HZ}]\label{Toeplitz type spectral}
Let $(\mathcal{A}, \mathcal{H}, \mathcal{D})$ be a spectral triple and $P\in\mathcal{B(H)}$ be an orthogonal projection. The quadruple $(\mathcal{A}, \mathcal{H}, \mathcal{D}, P)$ is of Toeplitz type if
\begin{enumerate}[$(i)$]
\item $P\mbox{ and }\mathcal{D}$ commute,
\item $[P\,,\,\pi(a)]\in\mathcal{C(H)}$ for all $a\in\mathcal{A}$.
\end{enumerate}
A Toeplitz type quadruple is said to be $P$-injective if $\mathrm{ker}(P\mathcal{D})\cap P\mathcal{H}=\{0\}$.
\end{dfn}

The following result is useful to determine when a spectral triple induces a Toeplitz type quadruple for a given orthogonal projection.

\begin{ppsn}[\cite{HZ}, Propn. 4.3]\label{Toeplitz check}
Let $(\mathcal{A}, \mathcal{H}, \mathcal{D})$ be a spectral triple and $P\in\mathcal{B(H)}$ be an orthogonal projection commuting with $\mathcal{D}$. Then, the following are equivalent:
\begin{enumerate}[$(i)$]
\item $[P\,,\,\pi(a)]\in\mathcal{C(H)}$ for each $a\in \mathcal{A}$,
\item $[P\mathcal{D}\,,\,\pi(a)]$ and $[(1-P)\mathcal{D}\,,\,\pi(a)]$ extend to bounded operators for each $a\in \mathcal{A}$,
\item $[(2P-1)\mathcal{D}\,,\,\pi(a)]$ extends to a bounded operator for each $a\in\mathcal{A}$.
\end{enumerate}
\end{ppsn}

\begin{dfn}[\cite{HZ}]
A spectral triple $(\mathcal{A}, \mathcal{H}, \mathcal{D})$ is called $P$-regular, where $P\in\mathcal{B(H)}$ is an orthogonal projection commuting with $\mathcal{D}$, if the above equivalent conditions hold (in which case, the quadruple $(\mathcal{A}, \mathcal{H}, \mathcal{D},P)$ is of Toeplitz type).
\end{dfn}

Note that all the above definitions can be recast equally in the operator system framework.


\subsection{Spectral triples for the extensions}\label{Sec3.3}

Consider a Toeplitz type extension
\[
0\longrightarrow\mathcal{K}(H_A)\otimes B\longrightarrow E\longrightarrow A\longrightarrow 0
\]
as in \Cref{Toeplitz extension}. Assume that $(\mathcal{A}, H_A, \mathcal{D}_A)$ and $(\mathcal{B}, H_B, \mathcal{D}_B)$ are spectral triples for $A$ and $B$ with faithful representations $\pi_A$ and $\pi_B$ respectively. Moreover, let $P \in B(H_A)$ be an orthogonal projection such that $(\pi_A, \pi_B, P)$ constitutes a Toeplitz triple (\Dref{Toeplitz extension}) and the quadruple $(\mathcal{A}, H_A, \mathcal{D}_A, P)$ is of Toeplitz type (\Dref{Toeplitz type spectral}).


Consider the Hilbert space $\mathcal{H}:=H_A\otimes H_B$. The representation $\pi:E\rightarrow\mathcal{B}(\mathcal{H})$ given by
\begin{align*}
\pi(T\otimes b)(\xi\otimes\eta) &:= TP(\xi)\otimes \pi_B(b)(\eta)\\
\pi(PaP\otimes 1)(\xi\otimes\eta) &:= P\pi_A(a)P(\xi)\otimes\eta
\end{align*}
where $b\in B,\,a\in A,$ and $T\in\mathcal{K}(PH_A)$, is faithful but degenerate (i.e. non-unital). We have another representation $\pi_\sigma:E\rightarrow\mathcal{B}(\mathcal{H})$ given by
\[
\pi_\sigma=\pi_{A}\circ\sigma\otimes 1
\]
where $\sigma:E\rightarrow A$ denotes the quotient map, and $A$ is viewed as a $C^*$-subalgebra of $B(H_A)$ by the faithful representation $\pi_A$. The representation $\pi_{\sigma}$ is non-degenerate (i.e. unital) but not faithful. Now consider the following representations
\[
\Pi_1\,,\,\Pi_2:E\longrightarrow B(\mathcal{H}\otimes\bbc^2)
\]
defined by
\[
\Pi_1:=\pi_\sigma\oplus\pi_\sigma\,\,\mbox{ and }\,\,\Pi_2:=\pi\oplus\pi_\sigma\,.
\]
Finally, consider the representation
\[
\Pi:E\to B\left((\mathcal{H}\otimes \bbc^2)^3\right)\quad\mbox{ defined by }\quad\Pi:= \Pi_{1} \oplus\Pi_{2} \oplus\Pi_{2}\,.
\]
Then, $\Pi$ is a faithful representation of $E$ on the Hilbert space $(\mathcal{H}\otimes \bbc^2)^3$.

Now, consider the dense unital $\ast$-subalgebra $\mathcal{E}$ of $E$ generated by
\begin{IEEEeqnarray}{lCl}\label{our domain}
\{k \otimes b\,:\,k\in \mathcal{C}(PH_{A}),\,b\in\mathcal{B}\}\quad\mbox{and}\quad\{PaP\otimes 1: a \in \mathcal{A}\}
\end{IEEEeqnarray}
and the following unbounded operators acting on $\mathcal{H}\otimes\bbc^2$, where $\mathcal{H}:=H_A\otimes H_B$,
\begin{align*}
\mathcal{D}_{1} &:= \begin{pmatrix}
D_A\otimes 1 & 1\otimes D_B\\
1\otimes D_B & -D_A\otimes 1
\end{pmatrix}\\
\mathcal{D}_{2} &:= \begin{pmatrix}
(1-P)D_A\otimes 1  & PD_A\otimes 1\\
PD_A\otimes 1 & -(1-P)D_A\otimes 1
\end{pmatrix}\\
\mathcal{D}_3 &:= 1\otimes D_B\otimes I_2
\end{align*}
each with domain $\mathbb{D}:=\mathrm{dom}(D_A)\odot\mathrm{dom}(D_B)\otimes\bbc^2$, where `$\odot$' denotes the algebraic tensor product. Next, using $\mathcal{D}_2$ and $\mathcal{D}_3$, construct the following unbounded operator
\[
\mathcal{D}_{I}:=\mathcal{D}_{2}\otimes\sigma_1+\mathcal{D}_3\otimes\sigma_2=\begin{pmatrix}
0 & \mathcal{D}_{2}-i\mathcal{D}_3\\
\mathcal{D}_{2}+i\mathcal{D}_3 & 0
\end{pmatrix}
\]
with domain $\mathbb{D}\oplus\mathbb{D}\subseteq(\mathcal{H}\otimes\bbc^2)\otimes\bbc^2$, where $\sigma_1$ and $\sigma_2$ are the standard Pauli spin matrices of order $2$. All of the above operators are essentially self-adjoint. Moreover, $\mathcal{D}_{1}$ and $\mathcal{D}_{I}$ have compact resolvents. If $\mathcal{D}_A\mbox{ and }\mathcal{D}_B$ are finitely summable, then so is $\mathcal{D}_1\mbox{ and }\mathcal{D}_I$. We refer to Lemma 4.5 in \cite{HZ} for the proof.

\begin{thm}[\cite{HZ}, Thm. 4.7]\label{2.11}
Let $(\mathcal{A},H_A,\mathcal{D}_A)$ and $(\mathcal{B},H_B,\mathcal{D}_B)$ be spectral triples over $C^*$-algebras $A\mbox{ and }B$ for a Toeplitz type extension
\[
0\longrightarrow\mathcal{K}(H_A)\otimes B\longrightarrow E\longrightarrow A\longrightarrow 0
\]
defined in \Cref{Toeplitz extension}. Then, $\big(\mathcal{E}, H_E:=(\mathcal{H}\otimes\bbc^2)^3, \mathcal{D}_E\big)$ is a spectral triple over $E$, where
\[
\mathcal{D}:=\begin{pmatrix}
\mathcal{D}_1 & 0 \\
0 & \mathcal{D}_I
\end{pmatrix}\,,
\]
and $\mathcal{E}$ is as defined in \Eref{our domain}. Moreover, the spectral dimension of this spectral triple is computed by the identity
\[
s_{0}(\mathcal{E}, H_E, \mathcal{D}_E)= s_{0}(\mathcal{A}, H_A, \mathcal{D}_A)+ s_{0}(\mathcal{B}, H_B, \mathcal{D}_B).
\]
\end{thm}
\noindent\textbf{Notation:} Throughout the article, we reserve the notation $(\mathcal{E}, H_E,\mathcal{D}_E)$ for the spectral triple over $E$ obtained by the Hawkins--Zacharias construction discussed above.
\smallskip

Recall \Dref{Toeplitz type spectral}. We have the following result from \cite{HZ}.

\begin{thm}[\cite{HZ}, Thm. 5.7]\label{3.9}
Consider a Toeplitz type extension
\[
0\longrightarrow\mathcal{K}(H_A)\otimes B\longrightarrow E\longrightarrow A\longrightarrow 0
\]
as in \Cref{Toeplitz extension}. Let $(\mathcal{A}, H_A, \mathcal{D}_A)$ and $(\mathcal{B}, H_B, \mathcal{D}_B)$ be spectral triples over $A\mbox{ and }B$ respectively such that there is an orthogonal projection $P\in \mathcal{B}(H_A)$ that makes the quadruple $(\mathcal{A}, H_A, \mathcal{D}_A, P)$ a Toeplitz type and $P$-injective. If the spectral triples $(\mathcal{A}, H_A, \mathcal{D}_A)$ and $(\mathcal{B}, H_B, \mathcal{D}_B)$ satisfy Rieffel's metric condition, then so does the spectral triple $(\mathcal{E}, H_E, \mathcal{D}_E)$ over $E$. In other words, if $(A, L_{\mathcal{D}_A})\mbox{ and }(B, L_{\mathcal{D}_B})$ are compact quantum metric spaces, then so is $(E, L_{\mathcal{D}_E})$.   
\end{thm}

Note that the condition of $P$-injectivity in the above theorem is a mild hypothesis, see discussion just before \cite[Remark $5.2$]{HZ}. We also remind that every unital $C^*$-algebra is by definition a complete operator system, therefore the main outcomes of the Hawkins--Zacharias construction (\Tref{2.11}, \ref{3.9}) fit into the operator system framework adopted throughout this paper.


\newsection{Unital $2$-contractive approximation of Toeplitz type}\label{Sec4}

Consider a Toeplitz type extension
\[
0\longrightarrow\mathcal{K}(H_A)\otimes B\longrightarrow E\longrightarrow A\longrightarrow 0
\]
as defined in \Cref{Toeplitz extension} and assume that \Tref{3.9} is at our disposal. In this section, we introduce an approximation notion that serves as the fundamental tool for analyzing the Lipschitz seminorms arising from the Hawkins--Zacharias spectral triples.

Recall that if $\mathcal{X}\subseteq A$ is an operator system and $B$ is a $C^*$-algebra, a linear map $\phi:\mathcal{X}\to B$ extends to $\phi_n:M_n(\mathcal{X})\to M_n(B)$ by $\phi_n\big((x_{ij})\big):=\big(\phi(x_{ij})\big)$. We call $\phi$ $n$-contractive if $\phi_n$ is contractive, i.e. $\|\phi_n\|\le 1$, and we call $\phi$ completely contractive if it is $n$-contractive for all $n$. See \cite{Paul} for more details.

Our first notion is that of a unital $2$-contractive approximation which is for general compact quantum metric spaces arising from spectral triples.
\begin{dfn}\label{def 1}
Let $(\mathcal{X}, H_X, \mathcal{D}_X)$ be a spectral triple for a complete operator system $X$ such that $(X, L_{\mathcal{D}_{X}})$ is a compact quantum metric space. Suppose that we have a sequence of complete sub-operator systems $\{X_n\}_{n\in\bbn}$ of $X$. A unital $2$-contractive approximation associated to the pair $(X,\{X_n\}_{n\in\bbn})$ is a sequence of unital $2$-contractive maps $\sigma_n:B(H_X) \longrightarrow B(H_X)$ satisfying the following:
\begin{enumerate}[$(i)$]
\item For each $n\in\bbn,\,\sigma_n(\mathcal{X}) \subseteq X_n \cap\mathcal{X} ;$
\item For all $\varepsilon>0$, there exists $N_\varepsilon\in\bbn$ such that $\|x-\sigma_n(x)\|\leq\varepsilon\, L_{\mathcal{D}_X}(x)$ for any $n\geq N_\varepsilon\mbox{ and } \forall\, x\in \mathcal{X};$
\item There exists a unitary $u\in B(H_X)$ such that $$\sigma_n(uxu^*)=u\sigma_n(x)u^*\quad\mbox{ and }\quad u[\mathcal{D}_X, \sigma_n(x)]u^*= \sigma_n(u[\mathcal{D}_X, x]u^*)$$
for all $\,x\in\mathcal{X}\mbox{ and } \forall\, n \geq 1$.
\end{enumerate}
\end{dfn}

\begin{rmrk}\rm
Unital $2$-contractive maps acting on $B(\mathcal{H})$ are not necessarily ucp maps, even if $\mathcal{H}$ is infinite dimensional (in fact, for every $n$ there are unital $n$-contractive maps that are not ucp\footnote{Thanks to Vern Paulsen for providing us explicit examples}), and hence the notion of $2$-contractive approximation defined in \Dref{def 1} is weaker than a similar notion of ucp approximation, while remaining sufficient for our purposes.
\end{rmrk}

\begin{xmpl}\rm\label{example} Consider $X=C(S^1)\subseteq B(L^2(S^1))$ with the canonical spectral triple $$\left(\mathcal{X}:=C^\infty(S^1)\,,\,L^2(S^1)\,,\,-i\frac{d}{d\theta}\right)$$ on it. For each $k\in\bbz$, consider $e_k:z\mapsto z^k,\,\forall\,z\in S^1$. For each $n\in\bbn$, define the following finite-dimensional subspace of $C(S^1)$:
\begin{IEEEeqnarray}{lCl}\label{triviality}
X_n:=\mathrm{span}\{e_k\,:k\in\bbz,\,|k|\leq n\}\,.
\end{IEEEeqnarray}
Each $X_n$ is unital $\ast$-invariant closed subspace of $X$, and therefore a complete sub-operator system of $X$. For each $n\in\bbn$, define $\,\sigma_n:C(S^1)\subseteq B(L^2(S^1))\longrightarrow B(L^2(S^1))$ by $\sigma_n(f):=F_n\star f$ for $f\in C(S^1)$, where $F_n(x):=\frac{\sin^2(nx/2)}{n\sin^2(x/2)}$ denotes the $n$-th Fej\'er kernel. These maps are ucp (see \cite[Thm 3.11]{Paul}), and hence extend to all of $B(L^2(S^1))$ by Arveson's extension theorem. The condition $(i)$ in \Cref{def 1} is automatically satisfied by \Eref{triviality}, and by Lemma $10$ in \cite{S}, we see that condition $(ii)$ is also satisfied. The unitary $u$ in condition $(iii)$ is $\mathrm{Id}\in B(L^2(S^1))$ in this case, and therefore the first part in condition $(iii)$ is trivially satisfied. Finally, the second part of the condition $(iii)$ with $u=\mathrm{Id}$ is satisfied by the fact that $\sigma_n(f^\prime)=\sigma_n(f)^\prime$ for any $f\in C^\infty(S^1)$ and for each $n\in\bbn$. Therefore, $\{\sigma_n\}_{n\in\bbn}$ determines a unital $2$-contractive approximation associated to the pair $(X,\{X_n\}_{n\in\bbn})$. 
\end{xmpl}

The above example extends to $C(\mathbb{T}^d)$; we postpone the details to \Sref{Sec6.3}.
\smallskip

While \Dref{def 1} applies to any spectral metric spaces, Toeplitz type extensions involve additional structure. The following refinement encodes precisely the extra structure required for lifting approximation through the Hawkins--Zacharias construction. Recall \Dref{differentiable compacts} and \ref{Toeplitz type spectral} in this regard.

\begin{dfn}\label{def 2}
Consider a Toeplitz type extension \Eref{Toeplitz type extension} and let $(\mathcal{A}, H_{A}, \mathcal{D}_{A},P)$ be a Toeplitz type quadruple over $A$ which is $P$-injective. Suppose that we have a sequence of complete sub-operator systems $\{A_n\}_{n\in\bbn}$ of $A$. A unital $2$-contractive approximation of Toeplitz type associated to the pair $(A,\{A_n\}_{n\in\bbn})$ is a sequence of unital $2$-contractive maps $\,\beta_n: B(H_{A}) \longrightarrow B(H_{A})$ satisfying the following:
\begin{enumerate}[$(i)$]
\item For each $n\ge 1,\,\beta_n(\mathcal{A}) \subseteq A_n\cap\mathcal{A} ;$
\item For all $\,\varepsilon>0$, there exists $\,N\in \bbn$ such that $\,\|a-\beta_n(a)\| \leq \varepsilon L_{\mathcal{D}_{A}}(a),\,\forall\, n\geq N\mbox{ and } \forall \,a \in \mathcal{A};$
\item There exists a unitary $v\in B(H_A)$ such that $v(Q_nP)=(Q_nP)v$ for all $n\in\bbn$, where $Q_n=\sum_{k=1}^n\,P_k$ and $P_k$'s are the spectral projections of $(P\mathcal{D}_{A}|_{PH_{A}})^{-1}$, and moreover the following conditions are satisfied:
\begin{enumerate}[$(a)$]
\item $\beta_n\big(\mathrm{Ad}_v([P\mathcal{D}_{A}\,,\,a])\big)=\mathrm{Ad}_v([P\mathcal{D}_{A}\,,\ \beta_n(a)])\,$ and $\,\beta_n\big(\mathrm{Ad}_v([(1-P)\mathcal{D}_{A}\,,\,a])\big)=\mathrm{Ad}_v([(1-P)\mathcal{D}_{A}\,,\,\beta_n(a)])$, for all $a\in\mathcal{A};$
\item $\beta_n(xP)=\beta_n(x)P$ and $\,\beta_n(Px)=P\beta_n(x)$, for all $x \in \mathcal{A} \cup \{\mathrm{Ad}_v([P\mathcal{D}_A\,,\,a]): a\in \mathcal{A}\};$
\item $\beta_n(TP)= Q_nTQ_nP$, for all $T \in \mathcal{C}(PH_{A});$ 
\item $\beta_n\big(v\big(\mathrm{cl}(TP\mathcal{D}_{A})\big)v^*\big)=v\big(\mathrm{cl}\big(\beta_n(TP)\mathcal{D}_{A}\big)\big)v^*$ and $\,\beta_n\big(v\big(\mathrm{cl}(\mathcal{D}_{A}TP)\big)v^*\big)=\\v\big(\mathrm{cl}\big(\mathcal{D}_{A}\beta_n(TP)\big)\big)v^*$, for all $T\in\mathcal{C}(PH_{A})$;
\end{enumerate}
where $\mathrm{Ad}_v(S):=vSv^*$ for any $S\in B(H_A)$.
\end{enumerate}
\end{dfn}
Observe that the condition $v(Q_nP)=(Q_nP)v$ implies that $vP=Pv,$ because $Q_n$ converges to $P$ in the strong operator topology.

Note that the above definition can also be recast equally in the operator system language. Also, with a slight modification in the Hawkins--Zacharias construction, the above approximation notion can be simplified. More precisely, condition $(d)$ in point $(iii)$ of \Dref{def 2} follows automatically if we modify the Hawkins--Zacharias construction slightly. We discuss it below.

Recall the subalgebra $\mathcal{C}(\mathcal{H})$ of $\mathcal{D}$-differentiable compacts associated with a spectral triple $(\mathcal{A},\mathcal{H},\mathcal{D})$ (\Dref{differentiable compacts}). 
\begin{dfn}
Given a spectral triple $(\mathcal{A},\mathcal{H},\mathcal{D})$, the subalgebra of continuously $\mathcal{D}$-differentiable compacts, denoted by $\mathcal{C}^1(\mathcal{H})$, consists of all $\mathcal{D}$-differentiable compacts $T\in\mathcal{C}(\mathcal{H})\subseteq\mathcal{K}(\mathcal{H})$ such that the closures $\mathrm{cl}(T\mathcal{D})\mbox{ and }\mathrm{cl}(\mathcal{D}T)$ are compact operators.
\end{dfn}

\begin{ppsn}\label{reqd density}
$\mathcal{C}^1(\mathcal{H})$ is dense in $\mathcal{K}(\mathcal{H})$.
\end{ppsn}
\begin{prf}
Since $(\mathcal{A}, \mathcal{H}, \mathcal{D})$ is a spectral triple,  $\mathcal{H}$ has an orthonormal eigenbasis $\{e_k\}_{k=1}^\infty$ for $\mathcal{D}$ with corresponding real eigenvalues $\{\lambda_k:k\in\bbn\}$ satisfying $|\lambda_k| \to \infty$. For each $n \in\bbn$, define $\widetilde{Q}_n$ to be the orthogonal projection onto the finite-dimensional spectral subspace
$$V_n := \mathrm{span}\{e_k : |\lambda_k| \leq n\} \subseteq \mathcal{H}.$$
Then, each $V_n$ is finite-dimensional, so $\widetilde{Q}_n$'s are finite-rank projections. Moreover, $V_n \subseteq \mathrm{Dom}(\mathcal{D})$ for each $n$, $\widetilde{Q}_n\mathcal{D}\xi = \mathcal{D}\widetilde{Q}_n\xi$ for all $\xi \in \mathrm{Dom}(\mathcal{D})$, and $\widetilde{Q}_n \to I$ strongly on $\mathcal{H}$. It is easy to see that each $\widetilde{Q}_n \in \mathcal{C}^1(\mathcal{H})$. We show that $\mathcal{C}^1(\mathcal{H})$ is dense in $\mathcal{C}(\mathcal{H})$. For this, take any $T \in \mathcal{C}(\mathcal{H})\subseteq\mathcal{K}(\mathcal{H})$ and define $T_n := \widetilde{Q}_nT\widetilde{Q}_n$ for each $n\in\bbn$.
\smallskip

\noindent\textbf{Step 1:} $T_n \in \mathcal{C}^1(\mathcal{H})$ for each $n\in\bbn$.
\smallskip

Fix any $n\in\bbn$. Observe that $T_n(\mathrm{Dom}(\mathcal{D}))\subseteq \mathrm{Dom}(\mathcal{D})$. For $\xi \in \mathrm{Dom}(\mathcal{D})$, we have
$$T_n\mathcal{D}\xi = \widetilde{Q}_nT\widetilde{Q}_n\mathcal{D}\xi = \widetilde{Q}_nT\mathcal{D}\widetilde{Q}_n\xi 
= \widetilde{Q}_n\,\mathrm{cl}(T\mathcal{D})\,\widetilde{Q}_n\xi,$$
where the last equality holds since $\widetilde{Q}_n\xi \in V_n \subseteq 
\mathrm{Dom}(\mathcal{D})$, so $\mathcal{D}\widetilde{Q}_n\xi$ 
is defined, and $\mathrm{cl}(T\mathcal{D})$ is the bounded extension of 
$T\mathcal{D}$ to all of $\mathcal{H}$. Therefore, $T_n\mathcal{D}$ extends to the bounded operator $\widetilde{Q}_n\,\mathrm{cl}(T\mathcal{D})\,\widetilde{Q}_n$ on $\mathcal{H}$. For $\mathcal{D}T_n$, take $\xi \in \mathrm{Dom}(\mathcal{D})$. Since $\widetilde{Q}_n\xi \in V_n \subseteq \mathrm{Dom}(\mathcal{D})$, and $T \in \mathcal{C}(\mathcal{H})$, we have $T\widetilde{Q}_n\xi \in \mathrm{Dom}(\mathcal{D})$. Hence $\mathcal{D}T\widetilde{Q}_n\xi$ 
is well defined, and
$$\mathcal{D}T_n\xi = \mathcal{D}\widetilde{Q}_nT\widetilde{Q}_n\xi = \widetilde{Q}_n\mathcal{D}T\widetilde{Q}_n\xi 
= \widetilde{Q}_n\,\mathrm{cl}(\mathcal{D}T)\,\widetilde{Q}_n\xi,$$
where the last equality holds since $\mathrm{cl}(\mathcal{D}T)$ is the 
bounded extension of $\mathcal{D}T$ to all of $\mathcal{H}$. Therefore, $\mathcal{D}T_n$ extends to the bounded operator $\widetilde{Q}_n\,\mathrm{cl}(\mathcal{D}T)\,\widetilde{Q}_n$ 
on $\mathcal{H}$. Since $\widetilde{Q}_n$ has finite rank, we get $T_n\in\mathcal{C}^1(\mathcal{H})$, which completes this step.
\medskip

\noindent\textbf{Step 2:} $\|T_n - T\| \to 0$.
\smallskip

First, we have
\begin{IEEEeqnarray}{lCl}\label{00}
\|T_n - T\| = \|\widetilde{Q}_nT\widetilde{Q}_n - T\|\leq\|\widetilde{Q}_nT(\widetilde{Q}_n - I)\|+\|(\widetilde{Q}_n - I)T\|.
\end{IEEEeqnarray}
Since $\widetilde{Q}_n-I\to 0$ strongly with $\sup_n\|\widetilde{Q}_n - I\| \leq 2$, and $T$ is a compact operator, we have that $\|(\widetilde{Q}_n - I)T\| \to 0$ as $n\to\infty$. To show that $\|\widetilde{Q}_nT(\widetilde{Q}_n - I)\| \to 0$ as $n\to\infty$, use the same reasoning together with the following fact $$\|\widetilde{Q}_nT(\widetilde{Q}_n - I)\| \leq \|T(\widetilde{Q}_n - I)\| = \|(\widetilde{Q}_n - I)T^*\|.$$
By \Eref{00}, now it follows that $\|T_n - T\| \to 0$, completing Step $2$.

Therefore, $\mathcal{C}^1(\mathcal{H})$ is dense in $\mathcal{C}(\mathcal{H})$. By the density of $\mathcal{C}(\mathcal{H})$ in $\mathcal{K}(\mathcal{H})$ \cite{HZ}, it follows that $\mathcal{C}^1(\mathcal{H})$ is dense in $\mathcal{K}(\mathcal{H})$.\qed
\end{prf}

Now, recall the definition of the unital dense $\ast$-subalgebra $\mathcal{E}$ of $E$ given in \Eref{our domain} in the Hawkins--Zacharias construction described in \Sref{Sec3.3}. Consider the unital $\ast$-subalgebra $\mathcal{E}^1$ of $E$ generated by
\begin{IEEEeqnarray}{lCl}\label{modified domain}
\{k \otimes b\,:\,k\in \mathcal{C}^1(PH_{A}),\,b\in\mathcal{B}\}\quad\mbox{and}\quad\{PaP\otimes 1: a \in \mathcal{A}\}\,.
\end{IEEEeqnarray}
Obviously, $\mathcal{E}^1\subseteq\mathcal{E}\subseteq E$ by construction since $\mathcal{C}^1(PH_A)\subseteq\mathcal{C}(PH_A)$.

\begin{ppsn}
The Hawkins--Zacharias construction remains valid if one replaces $\mathcal{E}$ by $\mathcal{E}^1$ defined in \Eref{modified domain}.
\end{ppsn}
\begin{prf}
The density of $\mathcal{E}^1$ in $E$ follows along the same line of the density of $\mathcal{E}$ in $E$, together with \Pref{reqd density}. \Tref{2.11} is clearly valid with $\mathcal{E}^1$ in place of $\mathcal{E}$, and restricting the domain of the Lipschitz seminorm to $\mathcal{E}^1$, the proof of \Tref{3.9} carries over verbatim.\qed
\end{prf}

\begin{crlre}\label{redundant}
The condition $(d)$ in point $(iii)$ of \Dref{def 2} follows from the condition $(c)$ if we replace $\mathcal{E}$ by $\mathcal{E}^1$ in the Hawkins--Zacharias construction.
\end{crlre}
\begin{prf}
Observe that the condition $(c)$ implies the equality for any $T\in\mathcal{K}(PH_A)$ by the density of $\mathcal{C}(PH_A)\mbox{ in }\mathcal{K}(PH_A)$ and the continuity of $\beta_n$. Thus for all $T\in\mathcal{K}(PH_A)$, we have $\beta_n(vTPv^*)=v\beta_n(TP)v^*$ for the unitary $v\in B(H_A)$ because $vQ_nP=Q_nPv\mbox{ and }vTv^*P\in\mathcal{K}(PH_A)$. Now, take $S\in\mathcal{C}^1(PH_A)$. Then, replacing $T$ in the above by $\mathrm{cl}(SP\mathcal{D}_A)$ and $\mathrm{cl}(\mathcal{D}_ASP)$ respectively, and using the fact that $Q_n$ commutes with $P\mathcal{D}_A$ for each $n\in\mathbb{N}$, we obtain the condition $(d)$.\qed
\end{prf}

Consequently, after replacing the dense $\ast$-subalgebra $\mathcal{E}$ by $\mathcal{E}^1$ in the Hawkins--Zacharias construction, the approximation hypothesis in \Dref{def 2} simplifies. It remains to exhibit an example that satisfies all the conditions in \Dref{def 2}, except possibly the condition $(d)$ in point $(iii)$.

\begin{xmpl}\rm\label{to refer end}
Let $B$ be any unital $C^*$-algebra and consider the short exact sequence
\[
0\longrightarrow\mathcal{K}(\ell^2(\bbn))\otimes B\longrightarrow E\longrightarrow \bbc\longrightarrow 0\,,
\]
the minimal unitization of $\mathcal{K}(\ell^2(\bbn))\otimes B$. This is a split extension with the Toeplitz projection $P=\mathrm{Id}\in B(\ell^2(\bbn))$. The spectral triple on the quotient under consideration is $(\bbc,\ell^2(\bbn),N-|e_0\rangle\langle e_0|)$, where $N:e_k\mapsto ke_k$ is the number operator on $\ell^2(\bbn)$. For each $n\in\bbn$, define $\beta_n:\mathcal{K}(\ell^2(\bbn))\oplus\bbc\subseteq B(\ell^2(\bbn))\longrightarrow B(\ell^2(\bbn))$ by $T+\lambda\mathrm{I}\longmapsto Q_n(T+\lambda\mathrm{I})Q_n+\lambda(I-Q_n)$, where $Q_n$'s are the sum of first $n$-many spectral projections of $(N-|e_0\rangle\langle e_0|)^{-1}$. Clearly, each $\beta_n$ is a ucp map, so it extends to a ucp (hence, unital $2$-contractive) map on $B(\ell^2(\bbn))$ by the Arveson's extension theorem. It is easy to verify that all the conditions in \Dref{def 2}, except the condition $(d)$ in point $(iii)$, are satisfied by taking $v=\mathrm{Id}\in B(\ell^2(\bbn))$.
\end{xmpl}

The following special case of \Dref{def 2} will be needed in \Sref{Sec6.2}.
\begin{dfn}\label{def 3}
Consider a Toeplitz type extension \Eref{Toeplitz type extension} and let $(\mathcal{A},H_{A},\mathcal{D}_{A},P)$ be a Toeplitz type quadruple over $A$ that is $P$-injective. A sequence $\{\rho_k\}_{k\in\bbn}$ of unital $2$-contractive maps $\rho_k:B(H_A)\to B(H_A)$ is said to be a Toeplitz type sequence of $2$-contractive maps if $\rho_k(\mathcal{A})\subseteq\mathcal{A}$ for each $k\in\bbn$, and moreover, conditions $(ii)$ and $(iii)$ in \Dref{def 2} are also satisfied.
\end{dfn}

Since the above definition is a special case of \Dref{def 2}, \Yref{redundant} is also applicable here.


\newsection{Lifting complete sub-operator systems to the extension}\label{Sec5}

Consider a Toeplitz type extension
\[
0\longrightarrow\mathcal{K}(H_A)\otimes B\longrightarrow E\longrightarrow A\longrightarrow 0.
\] 
as defined in \Cref{Toeplitz extension} such that $(E, L_{\mathcal{D}_E})$ is a compact quantum metric space obtained by the Hawkins--Zacharias construction (\Tref{3.9}). Let $\{A_n\}_{n\in\bbn}$ be a sequence of complete sub-operator systems of $A$. For each $n\in\bbn$, let $Q_n=\sum_{k=1}^nP_k$ where $P_k$ be the spectral projections of $\left(P\mathcal{D}_A|_{PH_A}\right)^{-1}\in\mathcal{K}(PH_A)$. For each $n\in\bbn$, consider the following subspace of $E$
\begin{IEEEeqnarray}{lCl}\label{quotient sub}
E_n:= Q_n\mathcal{K}(PH_A)Q_n\otimes B+ PA_nP\otimes\bbc\,.
\end{IEEEeqnarray}

On the other hand, let $\{B_n\}_{n\in\bbn}$ be a sequence of complete sub-operator systems of $B$. For each $n\in\bbn$, consider the following subspace of $E$
\begin{IEEEeqnarray}{lCl}\label{ideal sub}
G_n:= \mathcal{K}(PH_A) \otimes B_n+ PAP\otimes \mathbb{C}\,.
\end{IEEEeqnarray}

We show that the subspaces $E_n\mbox{ and }G_n$ of $E$ for each $n\in\bbn$, defined in \Eref{quotient sub} and \Eref{ideal sub} respectively, are \emph{complete} sub-operator systems of $E$. The faithful representation $\Pi$ of $E$ described in \Sref{Sec3.3} is the key ingredient in proving their closedness.

\begin{lmma}\label{S}
Let $S$ be a subspace of $\mathcal{K}(PH_A)\otimes B$ and $S_1$  be a closed subspace of $S$. Furthermore, let $V$ be a subspace of $A$ and $V_1$ be a closed subspace of $V$. Then $W_1= S_1+PV_1P\otimes \mathbb{C}$ is closed in $W= S+PVP\otimes \mathbb{C}$. 
\end{lmma}
\begin{prf}
Let $\{e_k\}_{k\in\bbn}$ defined by $e_k:= x_k +Pa_kP\otimes 1$ be a sequence in $W_1$ for $x_k\in S_1$ and $a_k\in V_1$ converging to $e=x+PaP\otimes 1\in W$ for $x\in S$ and $a\in V$. Considering the faithful representation $\Pi$ of $E$ described in \Sref{Sec3.3}, equivalently we can say that $\Pi(e_k-e)$ converges to $0$ in $B(H)$. But $\Pi=\Pi_1\oplus\Pi_2\oplus\Pi_2$ by construction. In particular, $\Pi_1(e_k-e)$ converges to $0$. Now the definition of $\Pi_1$ implies that $\|a_k-a\|$ converges to $0$ in $V$, which means that $a\in V_1$ since $V_1$ is closed in $V$. Obviously, $Pa_kP-PaP$ converges to $0$ as $P$ is orthogonal projection, and hence $x_k$ converges to $x$ in $S$. Since $S_1$ is closed in $S$, this implies that $x\in S_1$. Therefore, $e=x+PaP\otimes 1 \in W_1$ which shows that $W_1$ is closed in $W$.\qed
\end{prf}

\begin{lmma}\label{Q_m}
For each $n\in\bbn$, the subspace $Q_nK(PH_A)Q_n\otimes B$ is closed in $K(PH_A)\otimes B$, where $Q_n=\sum_{k=1}^nP_k$ and $P_k$'s are the spectral projections of $\left(P\mathcal{D}_A|_{PH_A}\right)^{-1}\in\mathcal{K}(PH_A)$.
\end{lmma}
\begin{prf}
Follows from the fact that $Q_n$'s are orthogonal projections.\qed
\end{prf}

\begin{ppsn}\label{pf of sub}
For each $n\in\bbn$, the subspaces $E_n\mbox{ and }G_n$ of $E$, defined in \Eref{quotient sub} and \Eref{ideal sub} respectively, are complete sub-operator systems of $E$.
\end{ppsn}
\begin{prf}
We show that both the subspaces $E_n$ and $G_n$ of $E$ are closed in $E$ for each $n\in\bbn$. To show for $E_n$, take $S=\mathcal{K}(PH_A)\otimes B,\,S_1= Q_n\mathcal{K}(PH_A)Q_n\otimes B,\,V=A\mbox{ and }V_1=A_n$ in Lemma \ref{S}. Then, $S_1$ is a closed subspace of $S$ by Lemma \ref{Q_m}, and $V_1$ is a closed subspace of $V$ (since $A_n\mbox{'s}$ are complete sub-operator systems of $A$). Similarly, to show for $G_n$, take $S= \mathcal{K}(PH_A)\otimes B,\,S_1= \mathcal{K}(PH_A)\otimes B_n,\mbox{ and }V_1=A=V$ in Lemma \ref{S}. Then, $S_1$ is a closed subspace of $S$ (since $B_n$'s are complete sub-operator systems of $B$). The conclusion now follows from Lemma \ref{S}. The unitality (the unit of $E$ is $P1_AP\otimes1_B$) and $\ast$-invariance are immediate. Hence, both $E_n\mbox{ and }G_n$ are complete sub-operator systems of $E$.\qed
\end{prf}


\newsection{Quantum Gromov--Hausdorff convergence for the extension}\label{Sec6}

We now prove the two lifting theorems announced in the Introduction. Throughout the section, we work under the following standing assumptions. We have a Toeplitz type extension 
\[
0\longrightarrow\mathcal{K}(H_A)\otimes B\longrightarrow E\longrightarrow A\longrightarrow 0
\]
as defined in \Cref{Toeplitz extension}, along with a spectral triple $(\mathcal{B},H_B,\mathcal{D}_B)$ over $B$ and a Toeplitz type and $P$-injective quadruple $(\mathcal{A},H_A,\mathcal{D}_A,P)$ over $A$. We have a spectral triple $(\mathcal{E},H_E,\mathcal{D}_E)$ over $E$, where $\mathcal{E}$ is as defined in \Eref{our domain}, obtained by the Hawkins--Zacharias construction so that $(E,L_{\mathcal{D}_E})$ is a compact quantum metric space by \Tref{3.9}.
\smallskip

The proof of the lifting theorems requires analyzing the explicit description of the Lipschitz seminorm coming from the Dirac operator on the extension. Hence, we first write the commutator $[\mathcal{D}_E\,,\,e]$ explicitly for an arbitrary element $e=x+PaP\otimes 1\in\mathcal{E}$, where $x\in K(PH_A)\otimes B$. Recall the representations $\pi,\pi_\sigma,\Pi_1,\Pi_2$ defined in \Sref{Sec3.3}. Then, we have the following:
\begin{IEEEeqnarray}{lCl}\label{commutator 1}
[\mathcal{D}_{1}\,,\,\Pi_1(e)]=\begin{pmatrix}
 [\mathcal{D}_A,a]\otimes 1 & 0\\
0 & -[\mathcal{D}_A,a]\otimes 1
\end{pmatrix}
\end{IEEEeqnarray}
\[
[\mathcal{D}_{2}\,,\,\Pi_2(e)]=\begin{pmatrix}
0 & P[P\mathcal{D}_A,a]\otimes 1 -x(P\mathcal{D}_A\otimes 1)\\
[P\mathcal{D}_A,a]P\otimes 1 +( P\mathcal{D}_A \otimes 1)x & -[(1-P)\mathcal{D}_A, a]\otimes 1
\end{pmatrix},
\]
\[
[\mathcal{D}_3\,,\,\Pi_2(e)]=\begin{pmatrix}
[1\otimes \mathcal{D}_B,x] & 0\\
0 & 0
\end{pmatrix}
\]
using the fact that $[P, \mathcal{D}_A]=0$. Therefore,
\[
[\mathcal{D}_{I}\,,\,\Pi_2(e)\oplus\Pi_2(e)]=\begin{pmatrix}
0 & [\mathcal{D}_{2}-i\mathcal{D}_3\,,\,\Pi_2(e)]\\
[\mathcal{D}_{2}+i\mathcal{D}_3\,,\,\Pi_2(e)] & 0
\end{pmatrix},
\]
where
\[
[\mathcal{D}_{2}-i\mathcal{D}_3\,,\,\Pi_2(e)]=\begin{pmatrix}
-i[1\otimes \mathcal{D}_B,x] &  P[P\mathcal{D}_A,a]\otimes 1 -x(P\mathcal{D}_A\otimes 1)\\
[P\mathcal{D}_A,a]P\otimes 1 +(P\mathcal{D}_A\otimes 1)x & -[(1-P)\mathcal{D}_A, a]\otimes 1
\end{pmatrix},
\]
and
\[
[\mathcal{D}_{2}+i\mathcal{D}_3\,,\,\Pi_2(e)]=\begin{pmatrix}
i[1\otimes \mathcal{D}_B,x] &  P[P\mathcal{D}_A,a]\otimes 1 -x(P\mathcal{D}_A\otimes 1)\\
[P\mathcal{D}_A,a]P\otimes 1 +(P\mathcal{D}_A\otimes 1)x & -[(1-P)\mathcal{D}_A, a]\otimes 1
\end{pmatrix}.
\]
Finally,
\begin{IEEEeqnarray}{lCl}\label{commutator 2}
[\mathcal{D}_E\,,\,\Pi_1(e)\oplus\Pi_2(e)\oplus\Pi_2(e)]=\begin{pmatrix}
[\mathcal{D}_{1},\Pi_1(e)] & 0\\
0 & [\mathcal{D}_{I},\Pi_2(e)\oplus\Pi_2(e)]
\end{pmatrix}
\end{IEEEeqnarray}
acting on the Hilbert space $H_E=(\mathcal{H}\otimes\bbc^2)^3\cong(\mathcal{H}\otimes\bbc^2)\oplus((\mathcal{H}\otimes\bbc^2)\otimes\bbc^2)$, where $\mathcal{H}=H_{A}\otimes H_{B}$.

The purpose of \Dref{def 1} and \Dref{def 2} was to encode exactly the compatibility needed with the commutator formula \Eref{commutator 2}.


\subsection{Approximations from the quotient part}\label{Sec6.1}

Consider a Toeplitz type extension and a pair $(A,\{A_n\}_{n\in\bbn})$ satisfying \Dref{def 2}. Assume that $A_n\cap\mathcal{A}$ is dense in $A_n$ for each $n\in\bbn$. Then, each $(A_n,L_{\mathcal{D}_A}|_{A_n})$ inherits a compact quantum metric structure by restricting the Lip-norm $L_{\mathcal{D}_A}$ on $A$ to $A_n$. The conditions in \Dref{def 2} now imply that the sequence $\{(A_n,L_{\mathcal{D}_A}|_{A_n})\}_{n\in\bbn}$ converges to $(A,L_{\mathcal{D}_A})$ in the quantum Gromov--Hausdorff distance, which follows from \Lref{Kaad}.

The following theorem is the principal lifting result from the quotient. It shows that controlled approximations of the quotient lift to approximations of the extension.
\begin{thm}\label{Quotient}
Consider a Toeplitz type extension 
\[
0\longrightarrow\mathcal{K}(H_A)\otimes B\longrightarrow E\longrightarrow A\longrightarrow 0
\]
such that $(E,L_{\mathcal{D}_E})$ is a compact quantum metric space obtained by the Hawkins--Zacharias construction. Let $\{A_n\}_{n\in\bbn}$ be a sequence of complete sub-operator systems of $A$ such that each $A_n\cap\mathcal{A}$ is dense in $A_n$. If there exists a unital $2$-contractive approximation of Toeplitz type associated to the pair $(A,\{A_n\}_{n\in\bbn})$, then the sequence of compact quantum metric spaces $\{(E_n,L_{\mathcal{D}_E}|_{E_n})\}_{n\in\bbn}$ consisting of complete sub-operator systems of $E$ given by $$E_n:= Q_n\mathcal{K}(PH_A)Q_n \otimes B+ PA_nP \otimes \bbc $$  converges to $(E,L_{\mathcal{D}_E})$ in the quantum Gromov--Hausdorff distance.
\end{thm}

First, let us verify that $(E_n, L_{\mathcal{D}_E}|_{E_n})$ is a compact quantum metric space for each $n\in\bbn$. By \Pref{pf of sub}, we have that each $E_n$ is a complete sub-operator system of $E$. Therefore, it is enough to show that $E_n\cap\mathcal{E}$ is dense in $E_n$ for each $n\in\bbn$, since $\ast$-invariance is obvious anyway.

The following density lemma is an easy consequence of the triangle inequality and density of the algebraic tensor product.
\begin{lmma}\label{dense}
Let $K_1, B_1\mbox{ and }A_1$ be closed subspaces of $\mathcal{K}(PH_A), B\mbox{ and }A$ respectively. Consider the norm-dense subspaces $K_2, B_2\mbox{ and }A_2$ of $K_1, B_1\mbox{ and }A_1$ respectively. Then, the subspace $J_2=K_2\odot B_2+PA_2P\otimes\bbc$ is norm-dense in $J_1=K_1\otimes B_1+PA_1P\otimes\bbc$, where `$\odot$' denotes the algebraic tensor product.
\end{lmma}

\begin{ppsn}\label{4.6}
We have $E_n\cap\mathcal{E}$ is norm-dense in $E_n$ for each $n\in\bbn$.
\end{ppsn}
\begin{prf}
Fix any $n\in\bbn$, and observe that $\mathcal{C}(PH_A)$ (recall \Dref{differentiable compacts}) contains all the spectral projections of $(P\mathcal{D}_A|_{PH_A})^{-1}$, and hence each $Q_n$. Hence, $Q_n\mathcal{C}(PH_A)Q_n\odot \mathcal{B}+P(A_n\cap\mathcal{A})P\otimes\mathbb{C}\subseteq E_n\cap\mathcal{E}$. Now in Lemma \ref{dense} take $K_1=Q_n\mathcal{K}(PH_A)Q_n,\,B_1=B,\,A_1=A_n,\,K_2=Q_n\mathcal{C}(PH_A)Q_n,\,B_2=\mathcal{B},\,A_2=A_n\cap \mathcal{A}$. Then, $B_2\mbox{ and }A_2$ are dense in $B_1\mbox{ and }A_1$ respectively (we have assumed that each $A_n\cap\mathcal{A}$ is dense in $A_n$ in \Tref{Quotient}). It remains only to verify that $K_2$ is dense in $K_1$. However, this follows from the fact that $\mathcal{C}(PH_A)$ is dense in $\mathcal{K}(PH_A)$ and $Q_n$'s are orthogonal projections. Therefore, by \Lref{dense} we have $Q_n\mathcal{C}(PH_A)Q_n\odot \mathcal{B}+P(A_n\cap\mathcal{A})P\otimes\mathbb{C}$ is dense in $E_n$, and hence $E_n\cap\mathcal{E}$ is dense in $E_n$.\qed
\end{prf}

The above Proposition, combined with Proposition \ref{pf of sub}, implies that $(E_n, L_{\mathcal{D}_E}|_{E_n})$ is a compact quantum metric space for each $n\in\bbn$. 
\medskip

Consider $Y:=(P\mathcal{D}_A|_{PH_A})^{-1}\in\mathcal{K}(PH_A)$ and let $P_k$ be the spectral projections of $Y$ (we use same notations used in \cite{HZ}). Set $Q_n:=\sum_{k=1}^n\,P_k$ and note that each $Q_n$ is a finite-rank projection on $PH_A$. Let $\{\beta_n\}_{n\in\bbn}$ be a unital $2$-contractive approximation of Toeplitz type associated to the pair $(A,\{A_n\}_{n\in\bbn})$ (\Dref{def 2}). Recall the faithful representation $\pi:E\to B(H_A\otimes H_B)$ (see \Sref{Sec3.3}) and note that $\pi(E)\subseteq B(H_A)\otimes B(H_B)$. For each $n\in\bbn$, define the following linear map
\begin{IEEEeqnarray}{lCl}\label{map0}
\alpha_n:B(H_A)\otimes B(H_B)\longrightarrow B(H_A\otimes H_B)\,,\qquad
\alpha_n:=\beta_n\otimes\mathrm{id}.
\end{IEEEeqnarray}
Note that $\alpha_n$'s defined in \Eref{map0} are continuous. By \Dref{def 2}, it follows that
\[
\alpha_n(x+PaP \otimes 1)=(Q_n\otimes 1)x(Q_n\otimes 1)+\beta_n(PaP)\otimes 1
\]
where $x\in\mathcal{K}(PH_A)\otimes B$ and $a\in A$.

\begin{lmma}\label{Image of alpha_n}
For each $n\in\bbn$, the image of $\mathcal{E}$ under the map $\alpha_n$ is contained in $E_n\cap\mathcal{E}$, where $$E_n= Q_n\mathcal{K}(PH_A)Q_n \otimes B+ PA_nP \otimes\bbc$$ and $\alpha_n$ is as defined in \Eref{map0}.
\end{lmma}
\begin{prf}
Fix any $n\in\bbn$, and recall that $Q_n\in\mathcal{C}(PH_A)$. Thus, if $T\in \mathcal{C}(PH_A)$, then $Q_nTQ_n\in \mathcal{C}(PH_A)$ as $\mathcal{C}(PH_A)$ is a subalgebra of $\mathcal{K}(PH_A)$. Since $\{\beta_n\}_{n\in\bbn}$ is a unital $2$-contractive approximation of Toeplitz type for the pair $(A, \{A_n\}_{n\in\bbn})$, we have $\beta_n(\mathcal{A})\subseteq A_n\cap\mathcal{A}$ and $\beta_n(PaP)=P\beta_n(a)P$ (see \Dref{def 2}), concluding the proof.\qed
\end{prf}

The immediate consequence is the following.
\begin{crlre}
The image of the map $\alpha_n$ restricted to $E$ is contained in $E_n.$
\end{crlre}
\begin{prf}
Since $\mathcal{E}$ is dense in $E$ and $\alpha_n(\mathcal{E})\subseteq E_n$ by \Lref{Image of alpha_n}. The result is then followed by using that $\alpha_n$ is continuous and $E_n$ is closed in $E$ by Proposition \ref{pf of sub}.\qed
\end{prf}

\begin{ppsn}\label{3.6}
For any $\varepsilon>0$, there exists $N\in\bbn$ such that $\forall\,n\geq N$, we have
\[
\|e-\alpha_n(e)\|\leq\varepsilon L_{\mathcal{D}_E}(e)
\]
for any $e\in \mathcal{E}$, where $\alpha_n$ is as defined in \Eref{map0}.
\end{ppsn}
\begin{prf}
Let $e=x+PaP \otimes 1 \in \mathcal{E}$, and $L_{\mathcal{D}_E}(e)\leq 1$. Then, from the proof of Lemma $5.4$ in \cite{HZ}, we have the following inequalities\,:
\begin{enumerate}[$(i)$]
\item $\|[\mathcal{D}_A, a]\|\leq 1$;
\item $\|(P\mathcal{D}_A \otimes 1)x\|\leq 3$;
\item $\|x(P\mathcal{D}_A \otimes 1)\|\leq 3$.
\end{enumerate}
Moreover, for $Y=(P\mathcal{D}_A|_{PH_A})^{-1}\in\mathcal{K}(PH_A)$, we also have the following\,:
\begin{IEEEeqnarray*}{lCl}
\|x-x(Q_n\otimes 1)\| &\leq& \|x(P\mathcal{D}_A\otimes 1)\|\,\|Y\otimes 1-YQ_n\otimes 1\|\\
&\leq& 3\|Y-YQ_n\|.
\end{IEEEeqnarray*}
Since $Q_n\to I_{PH_A}$ strongly with $\sup_n\|Q_n-I_{PH_A}\|\leq 2$, and $Y$ is a compact operator that commutes with $Q_n$, we have $\|Y-YQ_n\|=\|(I_{PH_A}-Q_n)Y\|\to 0$ as $n\to\infty$. Therefore, there exists $n_1\in\bbn$ such that
$$\|Y-YQ_n\|\leq\frac{\varepsilon}{12}$$ for all $n\geq n_1$. Then for all $n\geq n_1,$ we obtain the following:
\[
\|x-x(Q_n\otimes 1)\|\leq \frac{\varepsilon}{4}\,.
\]
Similarly, we have
\begin{IEEEeqnarray*}{lCl}
\|x-(Q_n\otimes 1)x\| &\leq& \|(Y-Q_nY)\|\,\|( P\mathcal{D}_A\otimes 1)x\|\\
&\leq& 3\|Y-Q_nY\|.
\end{IEEEeqnarray*}
For the similar reason as above, choose $n_2\in\bbn$ such that $\|Y-Q_nY\|\leq\frac{\varepsilon}{12}$ for all $n\geq n_2$. Then for all $n\geq n_2,$ we obtain
\[
\|x-(Q_n\otimes 1)x\|\leq \frac{\varepsilon}{4}\,.
\]
Therefore, for all $n\geq n_0:=\max\{n_1,n_2\}$, we obtain the following:
\begin{IEEEeqnarray}{lCl}\label{need here 1}
\|x-(Q_n\otimes 1)x(Q_n\otimes 1)\| &\leq& \|x-x(Q_n\otimes 1)\|+\|x-(Q_n\otimes 1)x\|\,\|Q_n\otimes 1\|\nonumber\\
&\leq& \frac{\varepsilon}{2}\, .
\end{IEEEeqnarray}
Since $\{\beta_n\}_{n\in\mathbb{N}}$ is a unital 2-contractive approximation of Toeplitz type, for all $n\geq 1$ and for all $a\in\mathcal{A}$, we have $\beta_n(Pa)=P\beta_n(a)\mbox{ and }\beta_n(aP)=\beta_n(a)P$. Then, we obtain the following:
\begin{IEEEeqnarray*}{lCl}
\|e-\alpha_n(e)\| &=& \|x+PaP \otimes 1 -\alpha_n(x+PaP \otimes 1 )\|\\
&=&\|x-(Q_n \otimes 1)x(Q_n \otimes 1) + PaP \otimes 1- \beta_n(PaP) \otimes 1\|\\
&=&\|x-(Q_n \otimes 1)x(Q_n \otimes 1) + PaP \otimes 1- P\beta_n(a)P \otimes 1\|\\
&\leq&\|x-(Q_n \otimes 1)x(Q_n \otimes 1)\|+ \|P(a-\beta_n(a))P \otimes 1 \|\\
&\leq&\|x-(Q_n \otimes 1)x(Q_n \otimes 1)\|+ \|a - \beta_n(a) \|\,.
\end{IEEEeqnarray*}
Again using the properties of $\{\beta_n\}_{n\in\mathbb{N}}$, there exists $n_3 \in \mathbb{N}$ such that for all $a \in \mathcal{A}$ and for all $n \geq n_3$, we have
\begin{IEEEeqnarray}{lCl}\label{need here 2}
\|a - \beta_n(a) \|\leq \frac{\varepsilon}{2}\,L_{\mathcal{D}_A}(a)=\frac{\varepsilon}{2}\,\|[\mathcal{D}_A, a]\|.
\end{IEEEeqnarray}
Because $\|[\mathcal{D}_A, a]\| \leq 1$, we get that
\begin{IEEEeqnarray*}{lCl}
\|e-\alpha_n(e)\| &\leq& \frac{\varepsilon}{2}+\frac{\varepsilon}{2}= \varepsilon
\end{IEEEeqnarray*}
for all $n\geq N:=\max\{n_0,n_3\}$ from \Eref{need here 1} and \Eref{need here 2}.

Finally, let $e \in \mathcal{E}$ be such that $L_{\mathcal{D}_E}(e) \neq 0$. Then, $\,L_{\mathcal{D}_E}\Big(\frac{e}{L_{\mathcal{D}_E}(e)}\Big)= 1$. Hence, the linearity of $\alpha_n$ implies that
\[
\|e-\alpha_n(e)\|\leq \varepsilon\,L_{\mathcal{D}_E}(e)
\]
for all $n\geq N$. Moreover, if $L_{\mathcal{D}_E}(e)=0$, then $e=\lambda P\otimes 1$ for some $\lambda \in \mathbb{C}$, and the unitality of $\beta_n$ implies that $\|e-\alpha_n(e)\|=0$. This concludes the proof.\qed
\end{prf}

By \Lref{Image of alpha_n}, we have $\alpha_n(e)\in\mathcal{E}$ for all $e\in\mathcal{E}$. Therefore, $L_{\mathcal{D}_E}(\alpha_n(e))$ makes sense, and we have the following result.
\begin{ppsn}\label{4.13}
For all $n\in\bbn$ and $e\in \mathcal{E}$, we have $L_{\mathcal{D}_E}(\alpha_n(e))\leq L_{\mathcal{D}_E}(e)$, where $\alpha_n$ is as defined in \Eref{map0}.
\end{ppsn}
\begin{prf}
Let $e=x +PaP \otimes1\in \mathcal{E}$ and  $x=\sum_{j=1}^kT_j\otimes b_j$ with $T_j\in\mathcal{C}(PH_A)$ and $b_j\in \mathcal{B}$. By \Eref{commutator 2}, we have
\[
L_{\mathcal{D}_E}(\alpha_n(e))=\max\,\{\|[\mathcal{D}_1\,,\,\Pi_1(\alpha_n(e))]\|\,,\,\|[\mathcal{D}_I\,,\,\Pi_2(\alpha_n(e))\oplus \Pi_2(\alpha_n(e))]\|\}.
\]
Using the commutator $[\mathcal{D}_1\,,\,\Pi_1(\alpha_n(e))]$ in \Eref{commutator 1}, the definition of $\alpha_n$ in \Eref{map0}, and the following facts
\begin{align*}
\beta_n(v[P\mathcal{D}_A, a]v^*) &= v[P\mathcal{D}_A, \beta_n(a)]v^*,\\
\beta_n(v[(1-P)\mathcal{D}_A, a]v^*) &= v[(1-P)\mathcal{D}_A, \beta_n(a)]v^*\,,
\end{align*}
(see \Dref{def 2}), we obtain the following:
\begin{IEEEeqnarray*}{lCl}
\|[\mathcal{D}_1\,,\,\Pi_1(\alpha_n(e))]\| &=& \|[\mathcal{D}_A\,,\,\beta_n(a)]\|\\
&=&\|v[\mathcal{D}_A\,,\,\beta_n(a)]v^*\|\\
&=& \|v[P\mathcal{D}_A\,,\,\beta_n(a)]v^*+v[(1-P)\mathcal{D}_A, \beta_n(a)]v^*\|\\
&=& \|\beta_n(v[P\mathcal{D}_A\,,\,a]v^*+v[(1-P)\mathcal{D}_A, a]v^*)\|\\
&=& \|\beta_n(v[\mathcal{D}_A\,,\,a]v^*)\|\,.
\end{IEEEeqnarray*}
The contractivity of $\beta_n$ ($\beta_n$'s are $2$-contractive, in particular) implies the following:
\[
\|[\mathcal{D}_1\,,\,\Pi_1(\alpha_n(e))]\|\leq \|v[\mathcal{D}_A\,,\,a]v^*\|=\|[\mathcal{D}_A\,,\,a]\|\leq L_{\mathcal{D}_E}(e)
\]
by \Eref{commutator 2}. Now, consider the unitary $V:=v\otimes 1\in B(H_A\otimes H_B)$. Then, using that (see \Dref{def 2})
\begin{align*}
Q_nPv &= vQ_nP,\,\,\forall\,n\geq 1,\\
\beta_n(TP) &= Q_nTQ_nP,\,\,\forall\,T\in \mathcal{C}(PH_A),
\end{align*} 
we have the following (recall the representation $\pi$ of $E$ from \Sref{Sec3.3}):
\begin{IEEEeqnarray*}{lCl}
& & V\big(\pm i[1\otimes\mathcal{D}_B\,,\,\pi(\alpha_n(x))]\big)V^*\\
&=& V\big(\pm i[1\otimes\mathcal{D}_B\,,\,(Q_n\otimes 1)x(Q_n\otimes 1)(P\otimes 1)]\big)V^*\\
&=& (v\otimes 1)\Big(\pm i\Big[1\otimes\mathcal{D}_B\,,\,(Q_n\otimes 1)\Big(\sum_{j=1}^kT_jP\otimes b_j\Big)(Q_n\otimes 1)\Big]\Big)(v^*\otimes 1)\\
&=&(v\otimes 1)\Big(\pm i\Big[1\otimes \mathcal{D}_B\,,\,\sum_{j=1}^kQ_nT_jPQ_n\otimes b_j\Big]\Big)(v^{*}\otimes 1)\\
&=&\sum_{j=1}^{k}(v\otimes 1)\big(\pm i \big(Q_nT_jQ_nP\otimes [\mathcal{D}_B\,,\, b_j]\big)\big)(v^{*}\otimes 1)\\
&=& \pm i\sum_{j=1}^k vQ_nT_jPQ_nv^*\otimes [\mathcal{D}_B\,,\,b_j]\\
&=&\pm i\sum_{j=1}^k vQ_nPT_jQ_nPv^*\otimes [\mathcal{D}_B\,,\,b_j]\hfill{(\mbox{as } T_j\in\mathcal{C}(PH_A))}\\
&=& \pm i\sum_{j=1}^k Q_nPvT_jPv^*Q_nP\otimes [\mathcal{D}_B\,,\,b_j]\hfill{\mbox{(by\,\,\Dref{def 2})}}\\
&=& \pm i \sum_{j=1}^k\beta_n(PvT_jPv^*P)\otimes [\mathcal{D}_B\,,\,b_j]\hfill{\mbox{(again by\,\,\Dref{def 2})}}\\
&=& \pm i \sum_{j=1}^k\beta_n(vT_jPv^*)\otimes [\mathcal{D}_B\,,\,b_j]\hfill{\mbox{(as $v$ commutes with $P$)}}\\
&=& (\beta_n\otimes id)\Big(\pm i(v\otimes 1)\Big(\sum_{j=1}^{k}T_jP\otimes [\mathcal{D}_B\,,\,b_j]\Big)(v^{*}\otimes 1)\Big)\\
&=& \alpha_n\Big((v\otimes 1)\Big(\pm i\Big[1\otimes\mathcal{D}_B\,,\,\sum_{j=1}^kT_jP\otimes b_j\Big]\Big)(v^*\otimes 1)\Big)\hfill{\mbox{(by \Eref{map0})}}\\
&=& \alpha_n\big(V\big(\pm i[1\otimes\mathcal{D}_B\,,\,\pi(x)]\big)V^*\big).
\end{IEEEeqnarray*}
Also, since $Q_n$ commutes with $P\mathcal{D}_A$, and by \Dref{def 2} we have $$\beta_n(vTP\mathcal{D}_Av^*)= v\beta_n(TP)\mathcal{D}_Av^*\,\forall,\,T\in \mathcal{C}(PH_A)\,,$$ we obtain the following using \Eref{map0}:
\begin{IEEEeqnarray*}{lCl}
V\alpha_n(x)(P\mathcal{D}_A\otimes 1)V^* 
&=& (v\otimes 1)(Q_n\otimes 1)x(Q_n\otimes 1)(P\mathcal{D}_A\otimes 1)(v^*\otimes 1)\\
&=&\sum_{j=1}^k vQ_nT_jQ_nP\mathcal{D}_Av^* \otimes b_j\\
&=& \sum_{j=1}^kv\beta_n(T_jP)\mathcal{D}_Av^*\otimes b_j\hfill{\mbox{(by \Dref{def 2})}}\\
&=& \sum_{j=1}^k\beta_n(vT_jP\mathcal{D}_Av^*)\otimes b_j\quad\hfill{\mbox{(again by \Dref{def 2})}}\\
&=&\alpha_n\Big((v\otimes 1)\Big(\sum_{j=1}^k T_jP \otimes b_j\Big)(P\mathcal{D}_A \otimes 1)(v^*\otimes 1)\Big)\\
&=&\alpha_n(V(x(P\mathcal{D}_A\otimes 1))V^*).
\end{IEEEeqnarray*}
Since by \Dref{def 2} we have $$\beta_n(v\mathcal{D}_ATPv^*)= v\mathcal{D}_A\beta_n(TP)v^*\,\forall,\,T\in \mathcal{C}(PH_A)\,,$$ along the same line as above, we obtain the following
\[
V(P\mathcal{D}_A\otimes 1)\alpha_n(x)V^*=\alpha_n(V((P\mathcal{D}_A\otimes 1)x)V^*).
\]
Again, using the fact that
\[
\beta_n(TP)=\beta_n(T)P\mbox{ and }\beta_n(PT)=P\beta_n(T),\,\,\forall\,T\in \{v[P\mathcal{D}_A, a]v^*: a\in \mathcal{A}\}
\]
(see \Dref{def 2}), we obtain the following equalities:
\begin{IEEEeqnarray*}{lCl}
V(P[P\mathcal{D}_A, \beta_n(a)]\otimes 1)V^*
&=& vP[P\mathcal{D}_A, \beta_n(a)]v^*\otimes 1\\
&=& Pv[P\mathcal{D}_A, \beta_n(a)]v^*\otimes 1\\
&=&P\beta_n(v[P\mathcal{D}_A, a]v^*)\otimes 1\nonumber\\
&=&\beta_n(Pv[P\mathcal{D}_A, a]v^*)\otimes 1\nonumber\\
&=& \alpha_n(vP[P\mathcal{D}_A, a]v^*\otimes 1)\\
&=& \alpha_n(V(P[P\mathcal{D}_A, a]\otimes 1)V^*).
\end{IEEEeqnarray*}
Along the same line as above, we also have the following:
\begin{align*}
V([P\mathcal{D}_A, \beta_n(a)]P \otimes 1)V^* &= \alpha_n(V([P\mathcal{D}_A, a]P \otimes 1)V^*),\\
V[(1-P)\mathcal{D}_A, \beta_n(a)]V^* \otimes 1
&= \alpha_n(V([(1-P)\mathcal{D}_A, a]\otimes 1)V^*).
\end{align*}
Let $V_1=V\otimes I_2$. Then $V_1$ is a unitary in $B(H_A\otimes\ H_B\otimes\bbc^2)$. Hence, by expanding the commutator $[\mathcal{D}_2\pm i\mathcal{D}_3\,,\,\Pi_2(\alpha_n(e))]$, we have
\begin{IEEEeqnarray*}{lCl}
& & \|[\mathcal{D}_2\pm i\mathcal{D}_3\,,\,\Pi_2(\alpha_n(e))]\|\\
&=& \Big{\|}\begin{pmatrix}
\pm i[1\otimes \mathcal{D}_B,\alpha_n(x)] & P[P\mathcal{D}_A, \beta_n(a)] \otimes 1-\alpha_n(x)(P\mathcal{D}_A \otimes 1)\\
 & \\
[P\mathcal{D}_A, \beta_n(a)]P \otimes 1+(P\mathcal{D}_A\otimes 1)\alpha_n(x) & -[(1-P)\mathcal{D}_A, \beta_n(a)] \otimes 1 
\end{pmatrix}\Big{\|}\\
&=& \Big{\|}V_1\begin{pmatrix}
 \pm i[1\otimes \mathcal{D}_B,\alpha_n(x)] & P[P\mathcal{D}_A, \beta_n(a)] \otimes 1-\alpha_n(x)(P\mathcal{D}_A \otimes 1)\\
 & \\
[P\mathcal{D}_A, \beta_n(a)]P \otimes 1+(P\mathcal{D}_A\otimes 1)\alpha_n(x) & -[(1-P)\mathcal{D}_A, \beta_n(a)] \otimes 1
\end{pmatrix}V_1^*\Big{\|}\\
&=& \Big{\|}\begin{pmatrix}
 V(\pm i[1\otimes \mathcal{D}_B,\alpha_n(x)])V^* & V(P[P\mathcal{D}_A, \beta_n(a)] \otimes 1-\alpha_n(x)(P\mathcal{D}_A \otimes 1))V^*\\
 & \\
V([P\mathcal{D}_A, \beta_n(a)]P \otimes 1+(P\mathcal{D}_A\otimes 1)\alpha_n(x))V^* & V(-[(1-P)\mathcal{D}_A, \beta_n(a)] \otimes 1)V^*
\end{pmatrix}\Big{\|}\\
&=& \Big{\|}\begin{pmatrix}
\alpha_n(V(\pm i[1\otimes \mathcal{D}_B,x])V^*) & \alpha_n(V(P[P\mathcal{D}_A, a] \otimes 1-x(P\mathcal{D}_A \otimes 1))V^*)\\
 & \\
\alpha_n(V([P\mathcal{D}_A, a]P \otimes 1+(P\mathcal{D}_A\otimes 1)x)V^*) & \alpha_n(V(-[(1-P)\mathcal{D}_A, a] \otimes 1)V^*) 
\end{pmatrix}\Big{\|}.\\
\end{IEEEeqnarray*}
The 2-contractivity of $\alpha_n$ finally implies the following\,:
\begin{IEEEeqnarray*}{lCl}
& & \|[\mathcal{D}_2\pm i\mathcal{D}_3\,,\,\Pi_2(\alpha_n(e))]\|\\
&\leq& \Big{\|}\begin{pmatrix}
V(\pm i[1\otimes \mathcal{D}_B,x])V^* & V(P[P\mathcal{D}_A, a] \otimes 1-x(P\mathcal{D}_A \otimes 1))V^*\\
 & \\
V([P\mathcal{D}_A, a]P \otimes 1+(P\mathcal{D}_A\otimes 1)x)V^* & V(-[(1-P)\mathcal{D}_A, a]\otimes 1)V^*
\end{pmatrix}\\
&=&\Big{\|}V_1\begin{pmatrix}
\pm i[1\otimes \mathcal{D}_B,x] & P[P\mathcal{D}_A, a] \otimes 1-x(P\mathcal{D}_A \otimes 1)\\
 & \\
[P\mathcal{D}_A, a]P \otimes 1+(P\mathcal{D}_A\otimes 1)x & -[(1-P)\mathcal{D}_A, a]\otimes 1
\end{pmatrix}V_1^*\Big{\|}\\
&=&\Big{\|}\begin{pmatrix}
\pm i[1\otimes \mathcal{D}_B,x] & P[P\mathcal{D}_A, a] \otimes 1-x(P\mathcal{D}_A \otimes 1)\\
 & \\
[P\mathcal{D}_A, a]P \otimes 1+(P\mathcal{D}_A\otimes 1)x & -[(1-P)\mathcal{D}_A, a]\otimes 1
\end{pmatrix}\Big{\|}\\
&\leq& L_{\mathcal{D}_E}(e)
\end{IEEEeqnarray*}
by \Eref{commutator 2}, which concludes the proof.\qed
\end{prf}

Note that the domain of our Lip-norm $L_{\mathcal{D}_E}$ is the dense $\ast$-subalgebra $\mathcal{E}\subseteq E$, and therefore for any complete sub-operator system $F\subseteq E$, the domain of $L_{\mathcal{D}_E}|_F$ is $F\cap \mathcal{E}.$
\medskip

\noindent\textbf{Proof of Theorem \ref{Quotient}:~}\label{reqd0}
Let $\varepsilon>0$. By \Pref{3.6}, there exists $N\in\mathbb{N}$ such that $\forall\, n\geq N$, we have $$\|e-\alpha_n(e)\|\leq\frac{\varepsilon}{2}\,L_{\mathcal{D}_E}(e)$$ $\forall\,e\in \mathcal{E}$, and by \Pref{4.13} we have $$L_{\mathcal{D}_E}(\alpha_n(e))\leq L_{\mathcal{D}_E}(e)\,.$$ By \Lref{Image of alpha_n}, we further have $\alpha_n(e)\in E_n\cap\mathcal{E}$ for all $e\in\mathcal{E}$, and therefore $L_{\mathcal{D}_E}(\alpha_n(e))=(L_{\mathcal{D}_E}|_{E_n})(\alpha_n(e))$. Now, since for each $n\in\bbn$, $E_n$ is a complete sub-operator system of $E$ by \Pref{pf of sub}, with $E_n\cap\mathcal{E}$ is dense in $E_n$ by \Pref{4.6}, the domain of $L_{\mathcal{D}_E}|_{E_n}$ is dense in $E_n$ so that each $(E_n,L_{\mathcal{D}_E}|_{E_n})$ is a compact quantum metric space. Therefore, applying \Lref{Kaad} we have 
\[
\mathrm{dist}_{Q}(E, E_n)\leq \frac{\varepsilon}{2}<\varepsilon
\]
$\forall\,n\geq N$, which concludes the proof.\qed


\subsection{Approximations from the ideal part}\label{Sec6.2}

In this subsection, we deal with the case where there is an approximation for the unital $C^*$-algebra underlying the stable ideal. Observe that for any pair $(B,\{B_n\}_{n\in\bbn})$ satisfying the conditions in \Dref{def 1}, with the additional hypothesis that each $B_n\cap\mathcal{B}$ is dense in $B_n$, we see that the sequence $\{(B_n,L_{\mathcal{D}_B}|_{B_n})\}_{n\in\bbn}$ converges to $(B,L_{\mathcal{D}_B})$ in the quantum Gromov--Hausdorff distance as an application of \Lref{Kaad}.

The following theorem is the principal lifting result from the ideal. It shows that controlled approximations of the algebra underlying the stable ideal lift to approximations of the extension. Recall \Dref{def 3} in this regard.
\begin{thm}\label{Ideal}
Consider a Toeplitz type extension 
\[
0\longrightarrow\mathcal{K}(H_A)\otimes B\longrightarrow E\longrightarrow A\longrightarrow 0
\]
such that $(E,L_{\mathcal{D}_E})$ is a compact quantum metric space obtained by the Hawkins--Zacharias construction. Let $\{B_n\}_{n\in\bbn}$ be a sequence of complete sub-operator systems of $B$ such that $B_n\cap\mathcal{B}$ is dense in $B_n$ for each $n\in\bbn$. Assume that there exists a Toeplitz type sequence of $2$-contractive maps on the quotient $A$. If there exists a unital $2$-contractive approximation associated to the pair $(B, \{B_n\}_{n\in\bbn})$, then the sequence of compact quantum metric spaces $\{(G_n,L_{\mathcal{D}_E}|_{G_n})\}_{n\in\bbn}$ consisting of complete sub-operator systems of $E$ given by $$G_n:=\mathcal{K}(PH_A)\otimes B_{n}+ PAP\otimes \mathbb{C}$$ converges to $(E,L_{\mathcal{D}_E})$ in the quantum Gromov--Hausdorff distance.
\end{thm}

First, let us verify that $(G_n,L_{\mathcal{D}_E}|_{G_n})$ is a compact quantum metric space for each $n\in\bbn$. Note that by \Pref{pf of sub}, each $G_n$ is a complete sub-operator system of $E$. Therefore, it is enough to show that $G_n\cap\mathcal{E}$ is dense in $G_n$ for each $n\in\bbn$, since $\ast$-invariance is obvious anyway.
\begin{ppsn}\label{4.5}
We have $G_n\cap\mathcal{E}$ is norm-dense in $G_n$ for each $n\in\bbn$.
\end{ppsn}
\begin{prf}
Fix any $n\in\bbn$. It is easy to see that $\mathcal{C}(PH_A)\odot (B_n\cap\mathcal{B} )+P\mathcal{A}P\otimes\mathbb{C}\subseteq G_n\cap\mathcal{E}$. Now take $K_1=\mathcal{K}(PH_A),\,B_1=B_n,\,A_1=A,\,K_2=\mathcal{C}(PH_A),\,B_2=B_n\cap \mathcal{B},\,A_2=\mathcal{A}$ in \Lref{dense}. Then $K_2, B_2\mbox{ and }A_2$ are dense in $K_1, B_1\mbox{ and }A_1$ respectively. Hence $\mathcal{C}(PH_A)\odot (\mathcal{B}\cap B_n)+P\mathcal{A}P\otimes\mathbb{C}$, so is $G_n\cap\mathcal{E}$ dense in $G_n$ by \Lref{dense}.\qed
\end{prf}

Assume that $\{\beta_m\}_{m\in\bbn}$ is a Toeplitz type sequence of 2-contractive maps (\Dref{def 3}), and consider the linear maps $\alpha_m:B(H_A)\otimes B(H_B) \longrightarrow B(H_A\otimes H_B)$ defined by $\alpha_m:=\beta_m\otimes\mathrm{id}$ similar to \Eref{map0}. For each $m,n\in\bbn$, consider the following subspaces of $E$\,:
\begin{IEEEeqnarray}{lCl}\label{hor and ver}
F_{m,n} &:=& \,Q_m\mathcal{K}(PH_A)Q_m \otimes B_n+ PAP \otimes \mathbb{C}\,,\nonumber\\
F_{m,\infty} &:=& \,Q_m\mathcal{K}(PH_A)Q_m \otimes B+ PAP \otimes \mathbb{C}\,.
\end{IEEEeqnarray}
We will show that each $F_{m, \infty}\mbox{ and }{F_{m, n}}$ are complete sub-operator systems of $E\mbox{ and }G_n$ respectively. Moreover, for each $m\in\bbn,\,F_{m, n}$ is a complete sub-operator system of $F_{m, \infty}$ for all $n\in\bbn$.

These complete sub-operator systems of the extension $E$ along the horizontal and vertical directions will play a major role in the sequel.

\begin{lmma}\label{4.16}
For each $m\in\bbn$, the subspace $F_{m, \infty}$ is closed in $E$.
\end{lmma}
\begin{prf}
Fix $m\in\bbn$ and take $S= \mathcal{K}(PH_A)\otimes B,\,S_1= Q_m\mathcal{K}(PH_A)Q_m\otimes B, \mbox{ and }V=A=V_1$ in \Lref{S}. Then $S_1$ is a closed subspace of $S$ by \Lref{Q_m}. The claim now follows from \Lref{S}.\qed
\end{prf}
\begin{lmma}\label{qw}
For each $m, n\in\bbn$, the subspace $F_{m, n}$ is closed in $G_n$.
\end{lmma}
\begin{prf}
Fix $m,n\in\bbn$ and take $S= \mathcal{K}(PH_A) \otimes B_n,\,S_1= Q_m\mathcal{K}(PH_A)Q_m\otimes B_n, \mbox{ and }V=A=V_1$ in Lemma \ref{S}. Then, $S_1$ is a closed subspace of $S$ by Lemma \ref{Q_m} (replacing $B$ by $B_n$). Claim now follows from Lemma \ref{S}.\qed
\end{prf}

\begin{lmma}\label{4.19}
For each $m, n\in\bbn$, the subspace $F_{m, n}$ is closed in $F_{m, \infty}$.
\end{lmma}
\begin{prf}
Fix $m,n\in\bbn$ and take $S=Q_m\mathcal{K}(PH_A)Q_m \otimes B,\,S_1= Q_m\mathcal{K}(PH_A)Q_m\otimes B_n$, and $V=A=V_1$ in Lemma \ref{S}. Since $Q_m$ is a finite-rank projection, there exists $k\in\bbn$ such that we have the following isomorphisms
\begin{align*}
Q_m\mathcal{K}(PH_A)Q_m\otimes B_n\, &\cong \,\mathbb{M}_k(B_n),\\
Q_m\mathcal{K}(PH_A)Q_m\otimes B\, &\cong \,\mathbb{M}_k(B).
\end{align*}
Since $B_n$ is closed in $B$, we have $S_1$ is a closed subspace of $S$. The claim now follows from Lemma \ref{S}.\qed 
\end{prf} 

Combining the lemmas (\ref{4.16}, \ref{qw}, \ref{4.19}), we finally have the following.
\begin{ppsn}\label{Fuzzy}
For each $m,n\in\mathbb{N},\,F_{m, \infty}\mbox{ and }{F_{m, n}}$ are complete sub-operator systems of $E\mbox{ and }G_n$ respectively. Moreover, for each $m\in\bbn,\,F_{m, n}$ is a complete sub-operator system of $F_{m, \infty}$ for all $n\in\bbn$. 
\end{ppsn}

\begin{lmma}\label{4.20}
We have $F_{m, \infty}\cap\mathcal{E}\mbox{ and }F_{m, n}\cap\mathcal{E}$ are norm-dense in $F_{m, \infty}\mbox{ and }F_{m, n}$ respectively, for each $m, n\in\bbn$.
\end{lmma}
\begin{prf}
Fix $m, n\in\bbn$. For $F_{m, \infty}\cap\mathcal{E}$, observe that $Q_m\mathcal{C}(PH_A)Q_m\odot\mathcal{B}+P\mathcal{A}P\otimes\bbc\subseteq F_{m, \infty}\cap\mathcal{E}$. Now take $K_1= Q_m\mathcal{K}(PH_A)Q_m,\,B_1= B,\,A_1= A, K_2=Q_m\mathcal{C}(PH_A)Q_m,\,B_2=\mathcal{B},\,A_2=\mathcal{A}$ in \Lref{dense}. Then $K_2,\,B_2\mbox{ and }A_2$ are dense in $K_1,\,B_1\mbox{ and }A_1$ respectively using the similar analysis in the proof of \Pref{4.6}. Hence $Q_m\mathcal{C}(PH_A)Q_m\odot\mathcal{B}+P\mathcal{A}P\otimes\bbc$, so is $F_{m, \infty}\cap \mathcal{E}$ dense in $F_{m, \infty}$ by \Lref{dense}.

For $F_{m, n}\cap\mathcal{E}$, notice that $Q_m\mathcal{C}(PH_A)Q_m\odot(B_n\cap\mathcal{B})+P\mathcal{A}P\otimes\bbc\subseteq F_{m, n}\cap\mathcal{E}$. Now take $K_1= Q_m\mathcal{K}(PH_A)Q_m,  B_1= B_n, A_1= A, K_2=Q_m\mathcal{C}(PH_A)Q_m, B_2= B_n\cap\mathcal{B}, A_2=\mathcal{A}$ in \Lref{dense}. Then $K_2, B_2\mbox{ and }A_2$ are dense in $K_1, B_1\mbox{ and }A_1$ respectively using the similar analysis in the proof of \Pref{4.6}. Hence $Q_m\mathcal{C}(PH_A)Q_m\odot (B_n\cap\mathcal{B})+P\mathcal{A}P\otimes\bbc$, so is $F_{m, n}\cap \mathcal{E}$ dense in $F_{m, n}$ by \Lref{dense}.\qed
\end{prf}

Combining \Pref{Fuzzy} and \Lref{4.20}, we see that $F_{m, \infty}\mbox{ and }F_{m, n}$ inherit compact quantum metric structures induced by $E$.

\begin{ppsn}\label{reqd0}
For all $\varepsilon>0$, there exists $m_0\in\bbn$ such that
\[
\mathrm{dist}_Q(E\,,\,F_{m,\infty}) < \frac{\varepsilon}{3}\quad\mbox{ and }\quad
\mathrm{dist}_Q(G_n\,,\,F_{m,n}) < \frac{\varepsilon}{3}
\]
for all $m\geq m_0$ and for all $n\in\bbn$.
\end{ppsn}
\begin{prf}
Recall the maps $\alpha_m$'s from \Eref{map0}. Observe that $\alpha_m(\mathcal{E})\subseteq F_{m, \infty}\cap\mathcal{E}$ and $\alpha_m(G_n\cap\mathcal{E})\subseteq F_{m, n}\cap\mathcal{E}$ for each $m,\,n\in\bbn,$ using the similar analysis in the proof of \Lref{Image of alpha_n} and using the fact that $\{\beta_n\}$ is a Toeplitz type sequence of $2$-contractive maps (\Tref{Ideal}). Now by \Pref {3.6}, choose $m_0\in\mathbb{N}$ such that $\forall\, m\geq m_0$, we have $\|e-\alpha_m(e)\|\leq\frac{\varepsilon}{4}\,L_{\mathcal{D}_E}(e),\,\forall\,e\in \mathcal{E}$. Then, for each $n \in \bbn$ and $e\in G_n \cap \mathcal{E}$, there exists $\alpha_m(e)\in F_{m,n}$ such that for all $m\geq m_0$, we have $$\|e-\alpha_m(e)\|\leq \frac{\varepsilon}{4}\,L_{\mathcal{D}_E}(e)$$ and moreover, $$(L_{\mathcal{D}_E}|_{F_{m, n}})(\alpha_m(e))\leq L_{\mathcal{D}_E}(e)$$ by \Pref{4.13}. Since, $F_{m, n}$ is a complete sub-operator system of $G_n$ (\Pref{Fuzzy}) and $F_{m, n}\cap\mathcal{E}$ dense in $F_{m, n}$ (Lemma \ref{4.20}), we can say that the domain of $L_{\mathcal{D}_E}|_{F_{m, n}}$ is dense in $F_{m, n}$. Also note that $(L_{\mathcal{D}_E}|_{G_n})|_{F_{m, n}}= L_{\mathcal{D}_E}|_{F_{m, n}}, \mbox{as}~F_{m, n}\subseteq G_n,$ which means that we have
\[\mathrm{dom}((L_{\mathcal{D}_E}|_{G_n})|_{F_{m, n}})=\mathrm{dom}(L_{\mathcal{D}_E}|_{F_{m, n}})=F_{m, n}\cap\mathcal{E},
\]
which is dense in $F_{m, n}$, and therefore $\mathrm{dom}((L_{\mathcal{D}_E}|_{G_n})|_{F_{m, n}})$ is dense in $F_{m, n}$. Now, an application of \Lref{Kaad} concludes the proof of $\mathrm{dist}_Q(G_n\,,\,F_{m,n}) < \frac{\varepsilon}{3}$. Similarly, the proof of $\mathrm{dist}_Q(E\,,\,F_{m, \infty}) < \frac{\varepsilon}{3}$ follows.\qed
\end{prf}

This concludes the analysis in the vertical direction.
\smallskip

Let $\{\sigma_n\}_{n\in\bbn}$ be a unital $2$-contractive approximation associated to the pair $(B,\{B_n\}_{n\in\bbn})$ (see \Dref{def 1}). Fix $m\in\bbn$ and for each $n\in\bbn,$ define the following linear maps
\begin{IEEEeqnarray}{lCl}\label{gamma}
\gamma_n:B(H_A)\otimes B(H_B)\rightarrow B(H_A\otimes H_B),\quad\gamma_n:=\mathrm{id}\otimes \sigma_n
\end{IEEEeqnarray}
Note that each $\gamma_n$ is continuous.

\begin{lmma}\label{4.23}
For each $m, n\in\bbn$, the image of $F_{m, \infty}\cap\mathcal{E}$ under the map $\gamma_n$ defined in \Eref{gamma} is contained in $F_{m, n}\cap\mathcal{E}$.
\end{lmma}
\begin{prf}
As $\{\sigma_n\}_{n\in\bbn}$ is a unital $2$-contractive approximation for the pair $(B, \{B_n\}_{n\in\bbn})$, we have for each $n\in\bbn,\,\sigma_n(\mathcal{B})\subseteq B_n\cap\mathcal{B}$. Hence, $\sigma_n(F_{m, \infty}\cap\mathcal{E})\subseteq F_{m, n}\cap\mathcal{E}.$  
\qed
\end{prf}
\begin{crlre}
For each $n\in\bbn$, the image of $F_{m, \infty}$ under the map $\gamma_n$ defined in \Eref{gamma} is contained in $F_{m, n}$.
\end{crlre} 
\begin{prf}
Fix any $n\in\bbn$. Since $F_{m, \infty}\cap\mathcal{E}$ is dense in $F_{m, \infty}$ by Lemma \ref{4.20}, and $\gamma_n(F_{m, \infty}\cap\mathcal{E})\subseteq F_{m, n}$ by Lemma \ref{4.23}, the result immediately follows by continuity of $\gamma_n$ and the closedness of $F_{m, n}$ in $E$ (because $F_{m, n}$ is closed in $F_{m, \infty}$ by \Lref{4.19}, and $F_{m, \infty}$ is closed in $E$ by \Lref{4.16}).\qed
\end{prf}

\begin{ppsn}\label{reqd1}
For each fixed $m\in\bbn$ and for all $\varepsilon>0$, there exists $N_{\varepsilon}\in\bbn$ such that for all $n\geq N_{\varepsilon}$, we have
\[
\|e-\gamma_n(e)\|\leq\varepsilon\,L_{\mathcal{D}_E}(e)
\]
for all $\,e\in F_{m,\infty} \cap \mathcal{E},$ where $\gamma_n$ is as defined in \Eref{gamma}.
\end{ppsn}
\begin{prf}
Let $e=x+PaP \otimes 1 \in F_{m,\infty} \cap \mathcal{E}$ and $x=\sum_{j=1}^kQ_mT_jQ_m\otimes b_j$ with $T_j\in\mathcal{K}(PH_A)$ and $b_j\in \mathcal{B}$. Observe that
\begin{IEEEeqnarray}{lCl}\label{qw1}
\|e-\gamma_n(e)\| &=& \|x+PaP \otimes 1-\gamma_n(x+PaP \otimes 1)\|\nonumber\\
&=& \|\sum_{j=1}^kQ_mT_jQ_m\otimes (b_j-\sigma_n(b_j))\|\nonumber\\
&=& \|(b_{pq}-\sigma_n(b_{pq}))_{pq}\|
\end{IEEEeqnarray}
where $\big(b_{pq}-\sigma_n(b_{pq})\big)_{pq}$ is a matrix of order equal to $\mathrm{rank}(Q_m)$ with entries in $\mathcal{B}$. This is possible working with the eigenbasis for $PH_A$ corresponding to the compact self-adjoint operator $(P\mathcal{D}_A|_{PH_A})^{-1}$, and the fact that unitary conjugation preserves norm. Then, we have
\[
\|\big(b_{pq}-\sigma_n(b_{pq})\big)_{pq}\|\leq C_m\max_{p,q}\,\{\,\|b_{pq}-\sigma_n(b_{pq})\|\,\}
\]
for $1\leq p,q\leq\mathrm{rank}(Q_m)$ and for some positive constant $C_m$ depending only on $\mathrm{rank}(Q_m)$. Using the fact that $\{\sigma_n\}_{n\in\bbn}$ is a $2$-contractive approximation for the pair $(B, \{B_n\}_{n\in\bbn})$, there exists $N_{\varepsilon}\in\bbn$ such that by \Dref{def 1} we have
\[
\|b_{pq}-\sigma_n(b_{pq})\|\leq\frac{\varepsilon}{2C_m}L_{\mathcal{D}_B}(b_{pq})
\]
for all $p,q$ and for all $n\geq N_{\varepsilon}$. From \Eref{qw1}, we get the following\,:
\begin{IEEEeqnarray*}{lCl}
\|e-\gamma_n(e)\| &\leq& C_m\,\frac{\varepsilon}{2C_m}\,\max_{p,q}\,\{L_{\mathcal{D}_B}(b_{pq}):1\leq p,q\leq\mathrm{rank}(Q_m)\}\\
&=& \frac{\varepsilon}{2}\,\max_{p,q}\,\{\|[\mathcal{D}_B\,,\,b_{pq}]\|:1\leq p,q\leq\mathrm{rank}(Q_m)\}\\
&\leq& \frac{\varepsilon}{2}\,\|([\mathcal{D}_B\,,\,b_{pq}])_{pq}\|\\
&\leq& \frac{\varepsilon}{2}\,\|i\,[1\otimes \mathcal{D}_B\,,\,\sum_j\,Q_mT_jQ_m\otimes b_j]\|\\
&\leq& \varepsilon L_{\mathcal{D}_E}(e)
\end{IEEEeqnarray*}
by the commutator expression in \eqref{commutator 2} for all $n\geq N_{\varepsilon}$ and $e\in F_{m,\infty}\cap\mathcal{E}$, which concludes the proof.\qed
\end{prf}

\begin{ppsn}\label{reqd2}
For all $n\in\bbn$ and $e\in F_{m,\infty}\cap\mathcal{E}$, we have $L_{\mathcal{D}_E}(\gamma_n(e))\leq L_{\mathcal{D}_E}(e)$, where $\gamma_n$ is as defined in \Eref{gamma}.
\end{ppsn}
\begin{prf}
Fix $n\in\bbn$ and let $e=x+PaP \otimes 1 \in F_{m,\infty} \cap \mathcal{E}$ and $x=\sum_{j=1}^kQ_mT_jQ_m\otimes b_j$, with $T_j\in\mathcal{K}(PH_A)$ and $b_j\in\mathcal{B}$. Now, by the commutator expression in \Eref{commutator 2}, we have
\[
L_{\mathcal{D}_E}(\gamma_n(e))=\max\,\{\,\|[\mathcal{D}_1, \Pi_1(\gamma_n(e))]\|\,,\,\|[\mathcal{D}_I, \Pi_2(\gamma_n(e))\oplus \Pi_2(\gamma_n(e))]\|\,\}.
\]
Expanding the commutator $[\mathcal{D}_1, \Pi_1(\gamma_n(e))]$ in \Eref{commutator 1} we get that
\begin{IEEEeqnarray*}{lCl}
\|[\mathcal{D}_1, \Pi_1(\gamma_n(e))]\|= \|[\mathcal{D}_A, a]\|\leq L_{\mathcal{D}_E}(e).
\end{IEEEeqnarray*}
Further, expanding the commutator $\|[\mathcal{D}_2\pm i\mathcal{D}_3\,,\,\Pi_2(\gamma_n(e))]\|$ yields the following\,:
\begin{IEEEeqnarray*}{lCl}
& & \|[\mathcal{D}_2\pm i\mathcal{D}_3\,,\,\Pi_2(\gamma_n(e))]\|\\
&=& \Big{\|}\begin{pmatrix}
\pm i[1\otimes \mathcal{D}_B, \gamma_n(x)] & P[P\mathcal{D}_A, a] \otimes 1-\gamma_n(x)(P\mathcal{D}_A \otimes 1)\\
 & \\
[P\mathcal{D}_A, a]P \otimes 1+(P\mathcal{D}_A\otimes 1)\gamma_n(x) & -[(1-P)\mathcal{D}_A, a]\otimes 1
\end{pmatrix} \Big{\|}\,.
\end{IEEEeqnarray*}
Consider the unitary $U:=1\otimes u\in B(H_A\otimes H_B)$. Using the fact that $u[\mathcal{D}_B, \sigma_n(b)]u^*= \sigma_n(u[\mathcal{D}_B, b]u^*)$ for all $b \in \mathcal{B}$ (\Dref{def 1}), we have
\begin{IEEEeqnarray*}{lCl}
U[1\otimes \mathcal{D}_B\,,\,\gamma_n(x)]U^* &=& U\Big[1\otimes \mathcal{D}_B\,,\, \gamma_n\Big(\sum_{j=1}^kQ_mT_jQ_m\otimes b_j\Big)\Big]U^*\\
&=& (1\otimes u)\Big[1\otimes \mathcal{D}_B\,,\,\sum_{j=1}^kQ_mT_jQ_m\otimes \sigma_n(b_j)\Big](1 \otimes u^*)\\
&=& \sum_{j=1}^kQ_mT_jQ_m\otimes u[\mathcal{D}_B, \sigma_n(b_j)]u^*\\
&=&\sum_{j=1}^kQ_mT_jQ_m\otimes \sigma_n(u[\mathcal{D}_B, b_j]u^*)\\
&=&\gamma_n\Big(\sum_{j=1}^kQ_mT_jQ_m\otimes u[\mathcal{D}_B, b_j]u^*\Big)\hfill{\mbox{(by \Eref{gamma})}}\\
&=&\gamma_n\big(U[1\otimes\mathcal{D}_B, x]U^*\big).
\end{IEEEeqnarray*}
Also, since for all $b\in \mathcal{B},\,u\sigma_n(b)u^*=\sigma_n(ubu^*)$ by \Dref{def 1}, we have the following:
\begin{IEEEeqnarray*}{lCl}
U\gamma_n(x)(P\mathcal{D}_A\otimes 1)U^* &=& (1\otimes u)\Big(\sum_{j=1}^kQ_mT_jQ_mP\mathcal{D}_A\otimes \sigma_n(b_j)\Big)(1\otimes u^*)\\
&=& \sum_{j=1}^kQ_mT_jQ_mP\mathcal{D}_A\otimes u\sigma_n(b_j)u^*\\
&=&\sum_{j=1}^kQ_mT_jQ_mP\mathcal{D}_A \otimes \sigma_n(ub_ju^*)\\
&=&\gamma_n\Big(\sum_{j=1}^kQ_mT_jQ_mP\mathcal{D}_A\otimes ub_ju^*\Big)\\
&=&\gamma_n(Ux(P\mathcal{D}_A\otimes 1)U^*).
\end{IEEEeqnarray*}
Similarly, we also have the following:
\[
U(P\mathcal{D}_A\otimes 1)\gamma_n(x)U^*= \gamma_n(U(P\mathcal{D}_A\otimes 1)xU^*).
\]
Using the fact that $\sigma_n$ is unital and $uu^*=1=u^*u$ by \Dref{def 1}, we have the following identities\,:
\begin{align*}
\gamma_n(U(P[P\mathcal{D}_A, a]\otimes 1)U^*) &= U(P[P\mathcal{D}_A, a]\otimes 1)U^*\,,\\
\gamma_n(U([P\mathcal{D}_A, a]P\otimes 1)U^*) &= U([P\mathcal{D}_A, a]P\otimes 1)U^*\,,\\
\gamma_n(U([(1-P)\mathcal{D}_A, a]\otimes 1)U^*) &= U([(1-P)\mathcal{D}_A, a]\otimes 1)U^*\,.
\end{align*}
Since $U_1= U \otimes I_2$ is a unitary in $B(H_A\otimes H_B\otimes\bbc^2)$, we have
\begin{IEEEeqnarray*}{lCl}
& & \|[\mathcal{D}_2\pm i\mathcal{D}_3\,,\,\Pi_2(\gamma_n(e))]\|\\ &=& \Big{\|}U_{1}\begin{pmatrix}
\pm i[1\otimes \mathcal{D}_B, \gamma_n(x)] & P[P\mathcal{D}_A, a] \otimes 1-\gamma_n(x)(P\mathcal{D}_A \otimes 1)\\
 & \\
[P\mathcal{D}_A, a]P \otimes 1+(P\mathcal{D}_A\otimes 1)\gamma_n(x) & -[(1-P)\mathcal{D}_A, a]\otimes 1
\end{pmatrix}U_{1}^{*}\Big{\|}\\
&=& \Big{\|}\begin{pmatrix} U(\pm i[1\otimes \mathcal{D}_B, \gamma_n(x)])U^* & U(P[P\mathcal{D}_A, a] \otimes 1-\gamma_n(x)(P\mathcal{D}_A \otimes 1))U^*\\
 & \\
U([P\mathcal{D}_A, a]P \otimes 1+(P\mathcal{D}_A\otimes 1)\gamma_n(x))U^* & U(-[(1-P)\mathcal{D}_A, a]\otimes 1)U^*
\end{pmatrix}\Big{\|}\\
&=& \Big{\|}\begin{pmatrix}
\gamma_n(U(\pm i[1\otimes \mathcal{D}_B, x])U^*) & \gamma_n(U(P[P\mathcal{D}_A, a] \otimes 1-x(P\mathcal{D}_A \otimes 1))U^*)\\
 & \\
\gamma_n(U([P\mathcal{D}_A, a]P \otimes 1+(P\mathcal{D}_A\otimes 1)x)U^*) & \gamma_n(U(-[(1-P)\mathcal{D}_A, a]\otimes 1)U^*)
\end{pmatrix}\Big{\|}\,.
\end{IEEEeqnarray*}
The $2$-contractivity of $\gamma_n=\mathrm{id}\otimes\sigma_n$ now implies the following\,:
\begin{IEEEeqnarray*}{lCl}
& & \|[\mathcal{D}_2 \pm i\mathcal{D}_3, \Pi_2(\gamma_n(e))]\|\\
&\leq&\Big{|}\Big{|}\begin{pmatrix}
U(\pm i[1\otimes \mathcal{D}_B, x])U^* & U(P[P\mathcal{D}_A, a] \otimes 1-x(P\mathcal{D}_A \otimes 1))U^*\\
 & \\
U([P\mathcal{D}_A, a]P \otimes 1+(P\mathcal{D}_A\otimes 1)x)U^* & U(-[(1-P)\mathcal{D}_A, a]\otimes 1)U^*\end{pmatrix}\\
&=& \Big{\|}U_1\begin{pmatrix}
\pm i[1\otimes \mathcal{D}_B, x] & P[P\mathcal{D}_A, a] \otimes 1-x(P\mathcal{D}_A \otimes 1)\\
 & \\
[P\mathcal{D}_A, a]P \otimes 1+(P\mathcal{D}_A\otimes 1)x & -[(1-P)\mathcal{D}_A, a]\otimes 1
\end{pmatrix}U_1^{*} \Big{\|}\\
&=& \Big{\|}\begin{pmatrix}
\pm i[1\otimes \mathcal{D}_B, x] & P[P\mathcal{D}_A, a] \otimes 1-x(P\mathcal{D}_A \otimes 1)\\
 & \\
[P\mathcal{D}_A, a]P \otimes 1+(P\mathcal{D}_A\otimes 1)x & -[(1-P)\mathcal{D}_A, a]\otimes 1
\end{pmatrix} \Big{\|}\\
&\leq& L_{\mathcal{D}_E}(e)
\end{IEEEeqnarray*}
by the commutator expression in \Eref{commutator 2}. Finally, observe the following\,:
\begin{IEEEeqnarray*}{lCl}
\|[\mathcal{D}_I, \Pi_2(\gamma_n(e))\oplus \Pi_2(\gamma_n(e))]\| &=& \max_\pm\{\|[\mathcal{D}_2 \pm i\mathcal{D}_3\,,\, \Pi_2(\gamma_n(e))]\|\}\\
&\leq& L_{\mathcal{D}_E}(e)\,,
\end{IEEEeqnarray*}
which concludes the proof.\qed
\end{prf} 

\begin{ppsn}\label{reqd3}
For all $m\in\bbn$, we have $\,\mathrm{dist}_Q(F_{m,n}\,,\,F_{m,\infty})\to 0$ as $n\to\infty$.
\end{ppsn}
\begin{prf}
Fix $m\in\bbn$ and let $\varepsilon>0$. By \Pref{reqd1}, there exists $N\in \bbn$ such that $\forall\, n \geq N$ and $\forall\,e\in F_{m, \infty}\cap \mathcal{E}$, we have $$\|e-\gamma_n(e)\|\leq\frac{\varepsilon}{2}\,L_{\mathcal{D}_E}(e),$$ and moreover $$L_{\mathcal{D}_E}(\gamma_n(e)) \leq L_{\mathcal{D}_E}(e)$$ by \Pref{reqd2}. Since for all $n\in\bbn\mbox{ and for all } e\in F_{m, \infty}\cap\mathcal{E}$, we have $\gamma_n(e)\in F_{m, n}\cap\mathcal{E}$ by \Lref{4.23}, it follows that $L_{\mathcal{D}_E}(\gamma_n(e))=(L_{\mathcal{D}_E}|_{F_{m, n}})(\gamma_n(e))$. Now, for all $n\in\bbn$, $F_{m, n}$ is a complete sub-operator system of $F_{m, \infty}$ (\Pref{Fuzzy}) with $F_{m, n}\cap\mathcal{E}$ dense in $F_{m, n}$ (\Lref{4.20}). Moreover, $(L_{\mathcal{D}_E}|_{F_{m, \infty}})|_{F_{m, n}}= L_{\mathcal{D}_E}|_{F_{m, n}},$ as $F_{m, n}\subseteq F_{m, \infty},$ which further implies the following: 
\[\mathrm{dom}((L_{\mathcal{D}_E}|_{F_{m, \infty}})|_{F_{m, n}})=\mathrm{dom}(L_{\mathcal{D}_E}|_{F_{m, n}})= F_{m, n}
\cap\mathcal{E}.
\]
Therefore, we have that $\mathrm{dom}((L_{\mathcal{D}_E}|_{F_{m, \infty}})|_{F_{m, n}})$ is dense in $F_{m, n}$. Finally, an application of \Lref{Kaad} shows that for all $n\geq N$, $$\mathrm{dist}_{Q}(F_{m, \infty}, F_{m,n})\leq\frac{\varepsilon}{2}<\varepsilon\,,$$ which concludes the proof.\qed
\end{prf}
\smallskip

\noindent\textbf{Proof of Theorem \ref{Ideal}:~} By \Pref{4.5}, each $(G_n,L_{\mathcal{D}_E}|_{G_n})$ is a compact quantum metric space. Now, let $\varepsilon>0$. By \Pref{reqd0}, there exists $m_0\in\bbn$ such that we have
\[
\mathrm{dist}_Q(G_n\,,\,F_{m_0,n}) < \varepsilon/3\quad\mbox{ and }\quad\mathrm{dist}_Q(E\,,\,F_{m_0,\infty}) < \varepsilon/3
\]
for all $n\in\bbn$. Also by \Pref{reqd3}, there exists $N_0\in\bbn$ such that
\[
\mathrm{dist}_Q(F_{m_0,n}\,,\,F_{m_0,\infty})< \varepsilon/3
\]
for all $n\geq N_0$. Therefore, by the triangle inequality we obtain $\mathrm{dist}_Q(G_n\,,\,E)<\varepsilon$ for all $n\geq N_0$, which concludes the proof.\qed


\subsection{Application}\label{Sec6.3}

\noindent(i)~\textit{The standard Podle\'s sphere}:
Consider the short exact sequence
\[
0\longrightarrow\mathcal{K}(\ell^2(\bbn))\longrightarrow E\longrightarrow \bbc\longrightarrow 0\,,
\]
the minimal unitization of $\mathcal{K}(\ell^2(\bbn))$. Recall Example \Eref{to refer end} with $B=\bbc$ and the standard spectral triple $(\bbc,\bbc,0)$ on it. Take the Toeplitz type quadruple $\big(\mathbb{C}, \ell^2(\bbn), N-|e_0\rangle\langle e_0|,\mathrm{Id}\big)$ on the quotient, and recall the dense $\ast$-subalgebra $\mathcal{E}^{1}\subseteq E$ from \Eref{modified domain} and \Yref{redundant}. Also, recall Example \Eref{to refer end} (with $B=\bbc$). Now, an application of \Tref{Quotient} shows that the sequence of the complete sub-operator systems $\{Q_n\mathcal{K}(\ell^2(\bbn))Q_n+\bbc\}_{n\in\bbn}$ converges to $E=\mathcal{K}(\ell^2(\bbn))\oplus \bbc$ in the quantum Gromov--Hausdorff distance, where $Q_n$'s are the sum of first $n$-many spectral projections of $(N-|e_0\rangle\langle e_0|)^{-1}$. It is well-known that the $C^*$-algebra $C(S_q^2)$ of the standard Podle\'s sphere is isomorhic to the minimal unitization of $\mathcal{K}(\ell^2(\bbn))$. Since each $Q_n$ is finite-rank, we see that the sequence $\{Q_n\mathcal{K}(\ell^2(\bbn))Q_n+\bbc\}_{n\in\bbn}$ of finite-dimensional sub-operator systems converges to the standard Podle\'s sphere in the quantum Gromov--Hausdorff distance.
\medskip

\noindent(ii)~\textit{Trigonometric polynomials for circle and their stabilization}:
Consider the following short exact sequence
\[
0\longrightarrow\mathcal{K}(\ell^2(\bbn))\otimes C(S^1)\longrightarrow E\longrightarrow \bbc\longrightarrow 0\,,
\]
the minimal unitization of $\mathcal{K}(\ell^2(\bbn))\otimes C(S^1)$, and the standard spectral triple $$\Big(C^{\infty}(S^1),\,L^2(S^1),\,-i\frac{d}{d\theta}\Big)$$ over $C(S^1)$. For each $n\in\bbn$, define $X_n:=\mbox{span}\{e_k:k\in\bbz,\,|k|\leq n\}$, where $e_k(z):=z^k$ for all $z\in S^1$, the space of trigonometric polynomials of degree at most $n$. Every $X_n$ is a complete sub-operator system of $C(S^1)$, and equip it with the Lip-norm obtained by restricting the Lip-norm on $C(S^1)$ induced by the spectral triple. Recall Example \Eref{example}. An application of \Lref{Kaad} shows that the sequence $\{X_n\}_{n\in\mathbb{N}}$ defined above converges to $C(S^1)$ in the quantum Gromov--Hausdorff distance. Furthermore, take the Toeplitz type quadruple $\big(\mathbb{C}, \ell^2(\bbn), N-|e_0\rangle\langle e_0|, \mathrm{Id}\big)$ on the quotient, and recall the dense $\ast$-subalgebra $\mathcal{E}^{1}\subseteq E$ from \Eref{modified domain} and \Yref{redundant}. Also, recall Example \Eref{to refer end} (with $B=C(S^1)$). Now, it follows from \Tref{Ideal} that the sequence $\{\mathcal{K}(\ell^2(\bbn))\otimes X_n+\mathbb{C}\}_{n\in\mathbb{N}}$ of complete sub-operator systems converges to the minimal unitization $E=\mathcal{K}(\ell^2(\bbn))\otimes C(S^1)\oplus\bbc$ in the quantum Gromov--Hausdorff distance. In other words, the convergence of trigonometric polynomials of degree at most $n$ to $C(S^1)$ in the quantum Gromov--Hausdorff distance is preserved under minimal unitization of the stabilization by $\mathcal{K}(\ell^2(\bbn))$.
\medskip

\noindent(iii)~\textit{Trigonometric polynomials for torus and their stabilization}:
Consider the following short exact sequence
\[
0\longrightarrow\mathcal{K}(\ell^2(\bbn))\otimes C(\mathbb{T}^d)\longrightarrow E\longrightarrow \bbc\longrightarrow 0\,,
\]
the minimal unitization of $\mathcal{K}(\ell^2(\bbn))\otimes C(\mathbb{T}^d)$, and the standard spectral triple $$\Big(C^{\infty}(\mathbb{T}^d),\, L^2(\mathbb{T}^d)\otimes\mathbb{C}^{2^{\lfloor{d/2}\rfloor}},\,D:=-i\sum_{\mu=1}^{d}\partial_{\mu}\otimes\gamma_{\mu}\Big)$$ over $C(\mathbb{T}^{d})$. For each $n\in\bbn$, define $X_n:=\mbox{span}\{e_k:k\in\bbz^d,\,|k|\leq n\}$, where $e_k(x):=e^{ik.x}$ for all $x\in\mathbb{T}^{d}$, which we call the space of trigonometric polynomials of degree at most $n$. Every $X_n$ is a complete sub-operator system of $C(\mathbb{T}^{d})$, and equip it with the Lip-norm obtained by restricting the Lip-norm on $C(\mathbb{T}^{d})$ induced by the spectral triple. For all $n\in\bbn$, define $$R_n:=\sigma_n\circ\rho_n:C^{\infty}(\mathbb{T}^d)\longrightarrow C^{\infty}(\mathbb{T}^d)\subseteq B\big(L^2(\mathbb{T}^d)\big)$$ where $\sigma_n\mbox{ and }\rho_n$ are unital contractive maps as defined in \cite{LS}. For each $n\in\bbn,\,R_n$ extends as a ucp map (see \cite[Thm 3.11]{Paul}) to $C(\mathbb{T}^d)$, and further to $B\big(L^2(\mathbb{T}^d)\big)$ by Arveson's extension theorem. From \cite{LS}, we have $$\|f-R_n(f)\|\leq c_n\|[D, f]\|,\,\forall n\in\mathbb{N}\mbox{ and } f\in C^{\infty}(\mathbb{T}^d).$$ Also, for any $f\in C^{\infty}(\mathbb{T}^d)$ we have $[\partial_{\mu}, M_f]= M_{\partial_{\mu}(f)}$ for all $\mu\in \{1,\ldots,d\}$, where $M_f$ denotes the multiplication operator acting on $L^2(\mathbb{T}^d)$. It is easy to see that for all $n\in\bbn\mbox{ and } \mu\in\{1,...,d\}$, we have $\partial_{\mu}(R_n(f))=R_n(\partial_{\mu}(f))$. Thus, it follows that $$[D,\, (R_n\otimes\mathrm{id})(M_f\otimes\mathrm{I}_d)]= (R_n\otimes \mathrm{id})\big([D,\, M_f\otimes\mathrm{I}_d]\big).$$ Since for each $n\in\mathbb{N},\,R_n\otimes\mathrm{id}$ is a ucp (in particular, unital 2-contractive) map from $B\big(L^2(\mathbb{T}^d)\otimes\mathbb{C}^{2^{\lfloor{d/2}\rfloor}}\big)$ to itself, we see that all the conditions in \Dref{def 1} are satisfied by taking the unitary $u=\mathrm{Id}\in B\big(L^2(\mathbb{T}^d)\otimes\mathbb{C}^{2^{\lfloor{d/2}\rfloor}}\big)$. An application of \Lref{Kaad} shows that the sequence $\{X_n\}_{n\in\mathbb{N}}$ defined above converges to $C(\mathbb{T}^d)$ in the quantum Gromov--Hausdorff distance. Furthermore, take the Toeplitz type quadruple $\big(\mathbb{C}, \ell^2(\bbn), N-|e_0\rangle\langle e_0|, \mathrm{Id}\big)$ on the quotient, and recall the dense $\ast$-subalgebra $\mathcal{E}^{1}\subseteq E$ from \Eref{modified domain} and \Yref{redundant}. Also, recall Example \Eref{to refer end} (with $B=C(\mathbb{T}^d)$). Now, it follows from \Tref{Ideal} that the sequence $\{\mathcal{K}(\ell^2(\bbn))\otimes X_n+\mathbb{C}\}_{n\in\mathbb{N}}$ of complete sub-operator systems converges to the minimal unitization $E=\mathcal{K}(\ell^2(\bbn))\otimes C(\mathbb{T}^d)\oplus\bbc$ in the quantum Gromov--Hausdorff distance. In other words, the convergence of trigonometric polynomials of degree at most $n$ to $C(\mathbb{T}^d)$ in the quantum Gromov--Hausdorff distance is preserved under minimal unitization of the stabilization by $\mathcal{K}(\ell^2(\bbn))$.
\medskip

\noindent(iv)~\textit{General stabilization for $2$-contractive approximation}: Let $B$ be any unital $C^*$-algebra and consider the following short exact sequence
\[
0\longrightarrow\mathcal{K}(\ell^2(\bbn))\otimes B\longrightarrow E\longrightarrow \bbc\longrightarrow 0\,.
\]
Assume that $B$ admits a spectral triple $(\mathcal{B},H_B,\mathcal{D}_B)$ that induce compact quantum metric structure on it. Suppose $B$ admits an approximation $\{B_n\}_{n\in\bbn}$ of complete sub-operator systems such that the conditions in \Dref{def 1} are satisfied, and $B_n\cap\mathcal{B}$ is dense in $B_n$ for each $n\in\bbn$. Similar to Examples $(ii)$ and $(iii)$ above, an application of \Tref{Ideal} shows that the sequence $\{\mathcal{K}(\ell^2(\bbn))\otimes B_n+\mathbb{C}\}_{n\in\bbn}$ of complete sub-operator systems converges to the minimal unitization $E=\mathcal{K}(\ell^2(\bbn))\otimes B\oplus\bbc$ in the quantum Gromov–Hausdorff distance, thus proving the stability of convergence under minimal unitization of the stabilization by $\mathcal{K}(\ell^2(\bbn))$. Examples $(ii)$ and $(iii)$ are special cases where $B=C(S^1)$ and $B=C(\mathbb{T}^d)$, respectively.

\bigskip

\noindent{\sc Vibhor Bhatt} (\texttt{vibhorbhatt25@iitk.ac.in})\\
{\footnotesize Department of Mathematics,
Indian Institute of Technology Kanpur,\\
Uttar Pradesh 208016, India}
\vspace*{.5cm}

\noindent{\sc Satyajit Guin} (\texttt{sguin@iitk.ac.in})\\
{\footnotesize Department of Mathematics,
Indian Institute of Technology Kanpur,\\
Uttar Pradesh 208016, India}

\vspace*{.5cm}

\noindent{\sc Bipul Saurabh} (\texttt{bipul.saurabh@iitgn.ac.in})\\
{\footnotesize Department of Mathematics,
Indian Institute of Technology Gandhinagar,\\
Palaj, Gandhinagar 382055, India}

\end{document}